\documentclass[a4paper, 11pt]{amsart}
\usepackage{amssymb,amsfonts,amsmath}
\usepackage{hyperref}
\usepackage{todonotes}
\usepackage{comment}
\usepackage{shuffle}
\usepackage[dvipsnames]{xcolor}

\usepackage{bm}
\usepackage{bbm}

\usepackage[margin=25mm]{geometry}

\usepackage[all,cmtip]{xy}
\usepackage{tikz}
\usetikzlibrary{matrix,arrows,decorations.pathmorphing, decorations.pathreplacing, angles,quotes, automata, matrix, positioning, calc, shapes.multipart}
\usepackage{tikz-cd}

 \DeclareFontFamily{U}{wncy}{}
\DeclareFontShape{U}{wncy}{m}{n}{<->wncyr10}{}
\DeclareSymbolFont{mcy}{U}{wncy}{m}{n}
\DeclareMathSymbol{\Sh}{\mathord}{mcy}{"58}

\newcommand{\parabolic}{\mathsf{P}}

\newcommand{\seed}{\Sigma}
\newcommand{\brick}{\mathbf{brick}}
\newcommand{\charlat}{\mathfrak{X}}

\newcommand{\Br}{\mathrm{Br}}

\newcommand{\SL}{\mathrm{SL}}

\newcommand{\borel}{\mathsf{B}}
\newcommand{\unipotent}{\mathsf{U}}
\newcommand{\flag}[1]{\mathsf{B}^{#1}}

\newcommand{\weave}{\mathfrak{w}}
\newcommand{\orich}[1]{\mathsf{R}^{\circ}_{#1}}
\newcommand{\rich}[1]{\mathsf{R}_{#1}}
\newcommand{\rpoly}[1]{\mathrm{R}_{#1}}

\newcommand{\rel}[1]{\xrightarrow{#1}}

\newcommand{\words}{\mathbb{W}}
\newcommand{\rootf}{\mathsf{r}}

\newcommand{\PP}{\mathbb{P}}

\newcommand{\conv}{\operatorname{conv}}

\newcommand{\supp}{\mathrm{supp}}
\newcommand{\word}[1]{\mathbf{#1}}
\newcommand{\wword}[1]{\mathbbm{#1}}
\DeclareMathOperator{\dem}{\delta}
\newcommand{\group}{\mathsf{G}}
\newcommand{\wo}{w_{0}}

\newcommand{\Pic}{\mathrm{Pic}}
\newcommand{\Cl}{\mathrm{Cl}}

\usepackage{mathrsfs}

\newcommand{\C}{\mathbb{C}}
\newcommand{\F}{\mathbb{F}}

\newcommand{\Z}{\mathbb{Z}}

\newcommand{\rank}{\mathrm{rank}}
\newcommand{\BS}{\mathrm{BS}}
\newcommand{\dynkin}{\mathsf{D}}

\newcommand{\wt}{\mathtt{w}}
\newcommand{\dil}[1]{\mathbb{T}_{#1}}
\newcommand{\frozen}[1]{d_{#1}}
\newcommand{\noness}[1]{I_{#1}}

\DeclareMathOperator{\Spec}{Spec}

\newtheorem{theorem}{Theorem}[section]

\newtheorem{proposition}[theorem]{Proposition}%[chapter]
\newtheorem{corollary}[theorem]{Corollary}%[chapter]
\newtheorem{lemma}[theorem]{Lemma}%[chapter]

\newtheorem{notation}[theorem]{Notation}

\theoremstyle{definition}
\newtheorem{remark}[theorem]{Remark}
\newtheorem{definition}[theorem]{Definition}
\newtheorem{example}[theorem]{Example}
\newtheorem{question}[theorem]{Question}

\makeatletter
\makeatletter
\def\@tocline#1#2#3#4#5#6#7{\relax
  \ifnum #1>\c@tocdepth % then omit
  \else
    \par \addpenalty\@secpenalty\addvspace{#2}%
    \begingroup \hyphenpenalty\@M
    \@ifempty{#4}{%
      \@tempdima\csname r@tocindent\number#1\endcsname\relax
    }{%
      \@tempdima#4\relax
    }%
    \parindent\z@ \leftskip#3\relax \advance\leftskip\@tempdima\relax
    \rightskip\@pnumwidth plus4em \parfillskip-\@pnumwidth
    #5\leavevmode\hskip-\@tempdima
      \ifcase #1
       \or\or \hskip 1em \or \hskip 2em \else \hskip 3em \fi%
      #6\nobreak\relax
    \dotfill\hbox to\@pnumwidth{\@tocpagenum{#7}}\par
    \nobreak
    \endgroup
  \fi}
\makeatother

\def\CC{\mathbb{C}}

\title[Torus actions on compactified  braid varieties]{Torus actions on compactified braid varieties and polytopality of subword complexes}

\author{Lara Bossinger}
\email{lara@im.unam.mx}
\address{Universidad Nacional Autónoma de México, Instituto de Matemáticas, Unidad Oaxaca, León 2, centro histórico, 68000 Oaxaca de Juárez, México}

\author{Mikhail Gorsky}
\email{mikhail.gorskii@univ-st-etienne.fr}
\address{Universit\"at Hamburg, Fachbereich Mathematik, Bundesstraße 55, 20146 Hamburg, Germany
\newline 
{\tiny{and}} Institut Camille Jordan UMR 5208, Université Jean Monnet, CNRS, Centrale Lyon, INSA Lyon, Université Claude Bernard Lyon 1, 20, rue Annino, 42023, Saint-Étienne, France}

\author{José Simental}
\email{simental@im.unam.mx}
\address{Universidad Nacional Autónoma de México, Instituto de Matemáticas. Ciudad Universitaria, CDMX, México}


\begin{document}

\begin{abstract}
Every cluster variety admits an action of its cluster dilation group. We prove that, in the case of braid varieties for simple Lie groups, this action always extends to a regular action on each of the brick compactifications. We explore two applications of this result. First, we show that any closed Richardson variety admits a faithful action of a torus of rank the Kazhdan-Lusztig $d$-invariant, answering affirmatively a recent question of E.~Gorsky--S.~Kim--M.~Sherman-Bennett. The same result holds for projected Richardson varieties. Second, we show that the braid variety is a torus if and only if
for each  of its brick compactifications, the polar dual of the moment polytope for this action realizes the corresponding subword complex. The braid words satisfying this property turn out to be precisely the double root free words of V.~Pilaud and C.~Stump. This provides a novel approach to the longstanding open question of the polytopality of spherical subword complexes asked by A.~Knutson and E.~Miller, and in particular gives infinite families of subword complexes admitting polytopal realizations in dimension higher than the rank of the corresponding Coxeter group. As a common consequence of these two applications, we classify all Bruhat intervals in finite crystallographic Coxeter groups which are isomorphic to face lattices of convex polytopes via certain double root free words.
\end{abstract}

\maketitle

\section{Introduction}

\subsection{Actions on cluster varieties} Let $A$ be a Fomin--Zelevinsky cluster algebra \cite{FominZelevinsky-I}. By definition, $A$ comes equipped with a family of generators, called \emph{cluster variables}, which are organized into overlapping subsets of the same finite cardinality called \emph{clusters}. We call an automorphism $\varphi: A \to A$ a \emph{cluster dilation}\footnote{Cluster dilations are often called \emph{cluster automorphisms}, but this term also has a very different meaning, so we decided not to use it. The term \emph{cluster dilations} was coined in \cite{PMS}.} if it sends every cluster variable to a scalar multiple of itself. The set of all cluster dilations forms a group, called the \emph{cluster dilation group} of $A$, and denoted by $\dil{A}$. By the Laurent phenomenon, a cluster dilation is determined by its action in a single cluster, so that in fact $\dil{A}$ is a subgroup of a torus. 

We consider the associated \emph{cluster variety} $\mathcal{A}$. By this, we mean either $\Spec(A)$ or the union of tori given by the nonvanishing of monomials in clusters -- the difference between these two is measured by the \emph{deep locus} \cite{CGSS} and it will not play a role in this introduction. By definition, the group $\dil{A}$ acts on $\mathcal{A}$ faithfully and algebraically. We consider \emph{partial compactifications} $\overline{\mathcal{A}}$ of $\mathcal{A}$. These are quasi-projective varieties together with an embedding $\mathcal{A} \hookrightarrow \overline{\mathcal{A}}$ that identifies $\mathcal{A}$ with a dense, open subset of $\overline{\mathcal{A}}$. We are interested in the following question.

\begin{question}\label{question}
Let $\mathcal{A}$ be a cluster variety, and let $\mathcal{A} \hookrightarrow \overline{\mathcal{A}}$ be a partial compactification. Does the action of $\dil{A}$ on $\mathcal{A}$ extend to an algebraic action on $\overline{\mathcal{A}}$?    
\end{question}

Gross, Hacking, Keel and Kontsevich's framework \cite{GHKK} of compactifying cluster varieties naturally fits this question as is shown in \cite{BCMNC}:
they provide a recipe to construct minimal models of $\mathcal A$-cluster varieties (\emph{i.e.} unions of cluster tori) that are invariant under the action of the cluster dilation group or subtori of it. The key idea is compactifying a cluster variety by building a $\dil{A}$-graded ring with a theta basis using so called \emph{broken line convex sets} \cite{brokenline_convex}.
Grassmannians and partial flag varieties are examples of this construction.

However, the answer to Question \ref{question} is in general \emph{no}. For example, a torus $(\C^{\times})^{n}$ is naturally a cluster variety, with cluster dilation group $\dil{} = (\C^{\times})^{n}$ acting on itself by left multiplication. Thus, given a partial compactification $(\C^{\times})^{n} \hookrightarrow \overline{\mathcal{A}}$, the answer to Question \ref{question} is affirmative if and only if $\overline{\mathcal{A}}$ is in fact a toric variety. Nevertheless, our first main result says that Question \ref{question} has an affirmative answer in many cases of interest.

\begin{theorem}\label{thm:main-1-intro}
    Question \ref{question} has an affirmative answer in the following cases:
    \begin{enumerate}
        \item When $\mathcal{A}$ is a braid variety for a simple Lie group \cite{CGGLSS, GLSBS, GLSB}, and $\overline{\mathcal{A}}$ is a brick compactification \cite{escobar}. 
        \item When $\mathcal{A}$ is an open Richardson variety in the flag variety for a simple Lie group, and $\overline{\mathcal{A}}$ is the corresponding closed Richardson variety.
    \end{enumerate}
\end{theorem}

Our proof is rather short and explicit, despite the fact that the rank of $\dil{\mathcal{A}}$ may be fairly large, and it does not use any heavy machinery of the minimal model program and nontrivial computations of broken line convex sets. The result suggests that brick varieties and closed Richardson varieties might fit into the setup of \cite{GHKK, BCMNC}. The computations and combinatorics in the Richardson case are typically more complicated if one just looks at them on their own without considering their brick resolutions and the more general class of all braid varieties, so we expect the verification of this for brick varieties to be easier compared to closed Richardson varieties; we plan to address this in future research.

In the rest of the introduction, we provide context for and state several consequences of Theorem \ref{thm:main-1-intro}.

\subsection{Braid and brick varieties} Braid varieties are affine, smooth algebraic varieties that have recently gained prominence due to their connections to contact and symplectic geometry \cite{ACSHLLW, CasalsWeng, CGGS1}, categorical link invariants \cite{CGGS2, Trinh}, Lie theory \cite{CGGLSS}, character varieties \cite{Mellit-character}, mirror symmetry \cite{CGGLSS, CGSS}, among others. Examples of braid varieties include positroid varieties in the Grassmannian \cite{GalashinLam-positroid}, double Bott-Samelson varieties \cite{SW}, reduced double Bruhat cells \cite{FominZelevinsky-Bruhat}, augmentation varieties of certain Legendrian links \cite{CGGS1}, and open Richardson varieties in the flag variety \cite{CGGLSS}. 

Let us quickly review the definition of the braid variety, referring the reader to Section \ref{sec:braid-varieties} for details. We fix a simple algebraic group $\group$, a pair of opposite Borel subgroups $\borel_{+}, \borel_{-} \subseteq \group$, and let $\dynkin$ denote the Dynkin diagram of $\group$. We denote by $W$ the corresponding Weyl group, with generators $s_{i}, i \in \dynkin$, and by $\Br_W$ the braid group, with generators $\sigma_i$, $i \in \dynkin$. Given a word $\word{i} = i_1i_2\dots i_r$ in the alphabet $\dynkin$, the braid variety is defined to be
\[
X(\word{i}) = \{(\flag{0}, \dots, \flag{r}) \in (\group/\borel_{+})^{r+1} \mid \flag{0} = \borel_{+}, \flag{r} = \dem(\word{i})\borel_{+}, \flag{t-1} \rel{s_{i_t}} \flag{t} \; \text{for every} \; t = 1, \dots, r\},
\]
where $\dem(\word{i}) \in W$ is the Demazure product and $\rightarrow$ denotes relative positions of flags. We remark that, up to a canonical isomorphism, $X(\word{i})$ depends only on the braid $\sigma_{i_1}\cdots \sigma_{i_r}$ defined by the word $\word{i}$.

Recently, both \cite{CGGLSS} and \cite{GLSB} constructed a cluster structure on the coordinate algebra $\C[X(\word{i})]$, making $X(\word{i})$ a cluster variety (see also \cite{GLSBS}, that constructs a cluster structure assuming $\group = \SL_n$). By the main result of \cite{CGGSSBS}, these cluster structures coincide. We denote by $\dil{\word{i}}$ the corresponding cluster dilation group. The cluster structure on $X(\word{i})$ is locally acyclic and really full rank (cf. \cite[Theorem 7.13, Corollary 8.5]{CGGLSS}) so the cluster dilation group $\dil{\word{i}}$ is a torus of rank equal to the number of frozen variables, cf. \cite[Proposition 5.1]{LamSpeyer-I}.

The brick variety $\brick(\word{i})$ was introduced by Escobar in \cite{escobar}, see also \cite[Appendix A]{KLS_projections}. The definition of the brick variety is similar to that of $X(\word{i})$, except that, instead of requiring $\flag{t-1} \rel{s_{i_t}} \flag{t}$, we require that either $\flag{t-1} = \flag{t}$ or $\flag{t-1} \rel{s_{i_t}} \flag{t}$. The brick variety $\brick({\word{i}})$ is a smooth projective variety \cite[Theorem 20]{escobar}, and by definition we have an open embedding $X(\word{i}) \hookrightarrow \brick(\word{i})$, so that $\brick(\word{i})$ is a compactification of $X(\word{i})$. 

We remark that the maximal torus $T = \borel_{+}\cap \borel_{-}$ naturally acts on $X(\word{i})$ and this action is by cluster dilations \cite[Corollary 4.26]{CGSS}. The action of $T$ extends naturally to an action on $\brick(\word{i})$. It has nice properties: for example, it is Hamiltonian, and the moment polytope is known \cite{escobar} to be the \emph{brick polytope} of the word $\word{i}$ introduced by Pilaud and Stump \cite{pilaud_stump_15} (generalizing the construction of \cite{pilaud_santos_12} for $\group = \SL_n$). However, the torus $\dil{\word{i}}$ can be significantly larger than $T$. For example, for $\group = \SL_n$, the torus $T$ has rank $n-1$, while $\dil{\word{i}}$ can have dimension up to $\binom{n}{2}$, see \cite[Proposition 6.3]{CGGLSS} and infinite families of examples in Section~\ref{sect:examples}, so Theorem \ref{thm:main-1-intro}(a) says that $\brick(\word{i})$ can have many symmetries that do not come from symmetries of the flag variety itself. This allows to decrease the upper bound on the minimal complexity of $\brick(\word{i})$ as a $\mathsf{T}$-variety\footnote{That is, a normal variety $X$ over $\CC$ with a faithful algebraic action of some algebraic torus $\mathsf{T}$, which has no relation to $T$ in general. Its complexity is the difference $\dim(X) - \rank(\mathsf{T})$. The trivial torus always acts faithfully with the maximal possible complexity $\dim(X)$, so one reasonably looks for actions of tori of higher rank, equivalently, of lower complexity.} from $\dim(\brick(\word{i})) - \rank (T)$ to the number of mutable variables in any given cluster seed in $\CC[X(\word{i})]$. $\mathsf{T}$-varieties naturally generalize toric varieties (which have complexity $0$ by definition), and for varieties of low minimal complexity a lot of tools which are standard and useful in the study of toric varieties admit at least partial generalizations (see \cite{T_varieties} and references therein), so decreasing such an upper bound should prove useful.

We have the following properties of the action $\dil{\word{i}} \curvearrowright \brick(\word{i})$.

\begin{theorem} \label{thm:intro_properties_of_the_action}
The group $\dil{\word{i}}$ acts algebraically and faithfully on $\brick(\word{i})$ with finitely many fixed points. Moreover,
\begin{enumerate}
    \item The fixed points are in bijection with the subwords $\word{j} = i_{j_1}\cdots i_{j_{k}}$ of $\word{i}$ such that $s_{i_{j_1}}\cdots s_{i_{j_k}} = \dem(\word{i})$. 
    \item For every subword $\word{j}$ of $\word{i}$ with $\dem(\word{j}) = \dem(\word{i})$, we have an embedding $\brick(\word{j}) \hookrightarrow \brick(\word{i})$. The images of $\brick(\word{j})$ are $\dil{\word{i}}$-invariant, and their subset for $\word{j}$ of length $r - 1$ spans the Picard group $\Pic(\brick(\word{i}))$. 
    \item The moment polytope of $\brick(\word{i})$ with respect to the $\dil{\word{i}}$-action projects onto (a dilation of) the brick polytope of $\word{i}$.
\end{enumerate}
\end{theorem}

\subsection{Toric brick varieties and subword complexes} We explore applications of Theorem \ref{thm:main-1-intro} when the braid variety $X(\word{i})$ is a torus. In particular, in this case $\brick(\word{i})$ is a smooth toric variety, and $(\brick(\word{i}), \brick(\word{i})\setminus X(\word{i}))$ is a toric pair. The braid variety $X(\word{i})$ is a torus if and only if it does not have mutable variables. Somewhat surprisingly, we show that the braid words $\word{i}$ for which $X(\word{i})$ is a torus are precisely the \emph{double root free} words (see Definition~\ref{def:double root free}), that have already appeared in work of V. Pilaud and C. Stump \cite{PilaudStump-EL} on subword complexes and their flip graphs. 

\begin{theorem}\label{thm:toric-brick-intro}
    Let $\word{i}$ be a braid word. The following conditions are equivalent. 
    \begin{itemize}
\item[(a)] $X(\word{i})$ is an algebraic torus;
\item[(b)] $\mathbb{C}[X(\word{i})]$ is a cluster algebra with no mutable vertices;
\item[(c)] $\brick(\word{i})$ is toric with respect to the action of $\dil{\word{i}}$;
\item[(d)] $(\brick(\word{i}), \brick(\word{i}) \backslash X(\word{i}))$ is a toric pair;
\item[(e)] $\word{i}$ does not contain any subword $\word{j} = i_{j_1}\cdots i_{j_k}$ such that $s_{i_{j_1}}\cdots s_{i_{j_k}} = \dem(\word{i})$ and $\word{j}$ is not reduced;
\item[(f)] $\word{i}$ is double root free.
\end{itemize}
\end{theorem}

We remark that $\brick(\word{i})$ can be a toric variety even when $X(\word{i})$ is not a torus.

\begin{example}
\label{ex:intro_111}
    Consider the word $\word{i} = 111$. Note that $\word{i}$ manifestly does not satisfy the conditions (e), (f) of Theorem \ref{thm:toric-brick-intro}. In this case, the brick variety is $\brick({\word{i}}) \cong \mathbb{P}^1 \times \mathbb{P}^1$, while $X(\word{i}) \cong \C^2 \setminus \{xy = 1\}$, cf. \cite[Example 4.16]{CGGS2}, which is not a torus. In fact, $X(\word{i})$ is a cluster variety of type $A_1$ with one frozen variable, so $\dil{\word{i}} \cong \C^{\times}$, and the action is given by $t.(x,y) = (tx, t^{-1}y)$. The brick variety $\brick(\word{i})$ is toric with respect to the larger torus $(\C^{\times})^{2}$, but this torus does not preserve the braid variety $X(\word{i}) \hookrightarrow \brick(\word{i})$. 
\end{example}

The double root free condition is designed, in a reasonable sense, precisely to prevent the word $\word{i}$ from having the words $111$ as subwords of its projections to generalized parabolic subgroups of $\group$.

\bigskip

The equivalence (a) $\Longleftrightarrow$ (d) can be alternatively established using the notion of \emph{log Calabi--Yau pairs of cluster type} introduced in \cite{EFM}. We prove in Proposition~\ref{prop:brick_cluster_pair} that 
the pair $(\brick(\word{i}), \brick(\word{i}) \backslash X(\word{i}))$ is of cluster type for an arbitrary $\word{i}$, following a suggestion from the work \cite{EFMS}. By \cite[Theorem 1.10]{EFM}, this does imply that such a pair is toric if and only if $X(\word{i})$ is a torus, albeit in a rather indirect way. We stress out that the rest of Theorem~\ref{thm:toric-brick-intro} and the more general Theorem~\ref{thm:intro_properties_of_the_action} do not follow from this directly.

\bigskip

We apply Theorem \ref{thm:toric-brick-intro} to the problem of finding polytopal realizations of subword complexes. Recall that, given a braid word $\word{i} = i_1\cdots i_r$ and an element $w \in W$, the subword complex $\Delta(\word{i}, w)$ is the simplicial complex on the vertex set $[r]$ whose faces consist of those subwords $\word{j}$ of $\word{i}$ whose complement $\word{i} \setminus \word{j}$ contains a reduced word for $w$. These complexes have been introduced by Knutson and Miller \cite{knutson_miller_04, knutson_miller_05} to study Gr\"obner degenerations of matrix Schubert varieties. They proved several nice general properties of subword complexes. In particular for $w = \dem(\word{i})$, as a topological manifold with boundary, $\Delta(\word{i}, \dem(\word{i}))$ is shellable and homeomorphic to a sphere. In this case, the dual complex is precisely the intersection complex of divisors in the boundary of the brick compactification $\brick(\word{i})$ of the braid variety $X(\word{i})$. Subword complexes have been also related to the fibers of exponentiation maps to totally nonnegative spaces \cite{DHM}, representations of quivers over $\mathbb{F}_1$ and their Hall algebras \cite{GoLi}, Hopf algebras and comultiplication in the coordinate rings of unipotent groups and Mirković-Vilonen polytopes \cite{BergeronCeballos, Zijun1}, and finite type cluster algebras \cite{pilaud_stump_15, JLS}. 

In \cite[Question 6.4]{knutson_miller_04}, Knutson and Miller ask whether every spherical subword complex $\Delta(\word{i}, \delta(\word{i}))$ can be realized as a convex polytope. Many spherical subword complexes for words in Coxeter groups of rank 1 and 2 can be realized via simplices or certain cyclic polytopes, see \cite[Section 6]{cls}. Pilaud and Stump \cite{pilaud_stump_15} found realizations for a class of words called \emph{root independent}. In our terms, these are precisely the words for which the natural map from $T = \borel_+ \cap \borel_-$ to $\dil{\word{i}}$ is surjective. 

In this case, the work \cite{pilaud_stump_15} shows that the polar dual of the brick polytope of $\word{i}$ realizes $\Delta(\word{i}, \delta(\word{i}))$.
A very important and most studied family of brick polytopes in the root independent case are the \emph{generalized associahedra}, closely related to the structure of cluster algebras associated to orientations of $\dynkin$. In this case, many interesting polytopal realizations of subword complexes exist, see the survey \cite{PSZ} and references therein. 
Still, as discussed earlier, root independency is a very restrictive condition. 

A fairly ad hoc approach to the problem of polytopality has been attempted in \cite{manneville} following \cite{gorsky_edge, G_subword_3}: these works studied transformations of subword complexes induced by certain natural operations on words. Some of these transformations preserve polytopality, but some do not, and checking whether a given subword complex can be obtained from a polytopal subword complex using only polytopality-preserving operations of this kind is complicated and has brought only very limited success.

By \cite[Corollary 2.16]{cls}, polytopality of all spherical subword complexes is equivalent to polytopality of a special family called \emph{generalized multi-associahedra}, and a lot of work since then concentrated in this direction. For $k = 1$, generalized $k$-associahedra are the usual generalized associahedra. Beyond small rank, despite much effort, only few explicit examples and no infinite families of polytopal realizations of generalized $k$-associahedra with $k > 1$ have been found to date, see \cite{CRS} and references therein. We do not make progress on the problem of polytopal realizations of multi-associahedra. Instead, we get back to more general subword complexes and show that our results allow to find realizations via brick varieties in much larger generality than that of root independent words. Namely, we show the polytopality of double root free subword complexes.

\begin{theorem}\label{thm:subword-complex-intro}
A word $\word{i}$ is double root free, so that $\brick(\word{i})$ is toric with respect to the action of $\dil{\word{i}}$, if and only if the polar dual of the moment polytope of $\brick(\word{i})$ with respect to the $\dil{\word{i}}$-action realizes the subword complex $\Delta(\word{i}, \dem(\word{i}))$. 
\end{theorem}

In the root independent case, the action of the cluster dilation group $\word{i}$ coincides with the action of the maximal torus $T$, and Theorem \ref{thm:subword-complex-intro} follows from works of Escobar \cite{escobar} and Pilaud-Stump \cite{pilaud_stump_15}. We note that in case of generalized associahedra, a different approach to polytopal realizations first appeared in the physics literature \cite{ABHY} and then was generalized in terms of representation theory of finite-dimensional algebras \cite{BMCLDMTY, PPPP}\footnote{The relevant setup was further axiomatized in \cite{GNP2}: the fans whose polytopal realizations are of interest generalize \emph{$g$-vector fans} of cluster algebras and are known as \emph{$g$-vector fans of 0-Auslander extriangulated categories}, see \cite{BorveKaipel} for the most general framework to date.}. These realizations were recently interpreted as moment polytopes \cite{GT}. Our approach is the opposite: we start from a toric variety and then check that the polar dual of its moment polytope realizes the subword complex.

We note that Theorem~\ref{thm:subword-complex-intro} and Thereom~\ref{thm:intro_properties_of_the_action}.(3) together imply that for double root free words, the subword complexes $\Delta(\word{i}, \dem(\word{i}))$ admit polytopal realizations whose polar duals project onto brick polytopes. This property does not hold for arbitrary words, failing already in such small examples as $\word{i} = 1212121$, cf. \cite[Proposition 5.9]{pilaud_santos_12}. 

Further, for the word $\word{i} = 111$ from Example~\ref{ex:intro_111}, the brick variety is isomorphic to $\mathbb{P}^1 \times \mathbb{P}^1$, while the subword complex is the boundary complex of the triangle, so the toric variety, the polar dual of whose moment polytope can realize the subword complex, is $\mathbb{P}^2$. All of this suggests that double root free case is indeed the largest reasonable generality where brick varieties can be used to realize spherical subword complexes, at least without using nontrivial elementary transformations.

\subsection{Richardson varieties} Now let $v, w \in W$ be elements of the Weyl group of $\group$. These determine a Bruhat cell $C_w$ and an opposite Bruhat cell $C^v$ in the flag variety $\group/\borel_{+}$. The \emph{open Richardson variety is}
\[
\orich{v,w} = C^v \cap C_w \subseteq \group/\borel_{+},
\]
and the closed Richardson variety $\rich{v,w}$ is simply the closure of $\orich{v,w}$ inside $\group/\borel_{+}$. Obviously, $\rich{v,w}$ is a compactification of $\orich{v,w}$. We note that $\orich{v,w}$ is nonempty if and only if $v \leq w$ in the Bruhat order, and in this case $\rich{v,w} = \bigsqcup_{v \leq v' \leq w' \leq w} \orich{v',w'}$, so that open Richardson varieties form a natural stratification of the flag variety $\group/\borel_{+}$. Open Richardson varieties appear in many areas of Lie theory, for example, the Kazhdan-Lusztig $R$-polynomial \cite{KazhdanLusztig} is the point-count of $\orich{v,w}$ over the finite field $\mathbb{F}_q$, and the homology of open Richardson varieties computes extension groups between Verma modules, see e.g. \cite{GalashinLam-qt}. In fact, every open Richardson variety is isomorphic to a braid variety \cite[Theorem 3.14]{CGGLSS}, so in particular $\orich{v,w}$ admits a cluster structure. The rank of the cluster dilation group is in fact the Kazhdan-Lusztig $d$-invariant $d_{v,w}$ \cite{BarkleyGaetzLam, patimo}. We denote this group by $\dil{v,w}$. The following result is analogous to Theorem \ref{thm:intro_properties_of_the_action}. We remark, however, that closed Richardson varieties do not coincide with brick varieties: instead, brick varieties for certain words provide resolutions of closed Richardson varieties, which can be singular.

\begin{theorem}\label{thm:richardson-intro}
    The group $\dil{v,w}$ acts algebraically and faithfully on the closed Richardson variety $\rich{v,w}$ with finitely many fixed points. Moreover,
    \begin{enumerate}
        \item The fixed points are in bijection with the elements of the Bruhat interval $[v,w]$.
        \item For every $v \leq v' \leq w' \leq w$, we have an embedding $\rich{v',w'} \hookrightarrow \rich{v,w}$. The image of $\rich{v',w'}$ is $\dil{v,w}$-invariant. 
        \item The moment polytope of $\rich{v,w}$ with respect to the $\dil{v,w}$-action projects onto (a dilation of) the Bruhat interval polytope of $[v,w]$. 
    \end{enumerate}
\end{theorem}

As with Theorem \ref{thm:intro_properties_of_the_action}, the torus $T = \borel_{+} \cap\borel_{-}$ acts on $\orich{v,w}$ by cluster dilations, and the action of $T$ clearly extends to $\rich{v,w}$. However, the rank of the torus $\dil{v,w}$ can be significantly larger than that of $T$, cf. \cite{hypercubes}. So the closed Richardson variety $\rich{v,w}$ admits many intrinsic symmetries that do not come from symmetries of the flag variety. 
Note that Theorem \ref{thm:richardson-intro} answers \cite[Question 1.6]{GKSB} affirmatively, even beyond type $A$. 

\subsection{Toric Richardson varieties and Bruhat intervals} We use Theorem \ref{thm:richardson-intro} to characterize those pairs $(v,w)$ for which the closed Richardson variety $\rich{v,w}$ is toric, generalizing  \cite[Theorem 1.2]{GKSB} beyond type $A$. 

\begin{theorem}\label{thm:toric-richardson-intro}
Let $v \leq w$ in Bruhat order. Let $\word{i}(w)$ be a reduced word for $w$, and $\word{i}(v^{-1}w_0)$ a reduced word for $v^{-1}w_0$. The following are equivalent.
\begin{itemize}
    \item[(a)] The open Richardson variety $\orich{v,w}$ is a torus.
    \item[(b)] The closed Richardson variety $\rich{v,w}$ is toric.
    \item[(c)] The Bruhat interval does not contain a subinterval isomorphic to $S_3$.
    \item[(d)] The Bruhat interval $[v,w]$ is a lattice.
    \item[(e)] The interval poset $\mathrm{Int}_{v,w} := \{[v',w'] \mid [v', w'] \subseteq [v,w]\}\cup\{\emptyset\}$, ordered by inclusion, is a lattice.
    \item[(f)] The Bruhat interval $[v,w]$ is the face lattice of a convex polytope.
    \item[(g)] The interval poset $\mathrm{Int}_{v,w}$ is the face lattice of a convex polytope.
    \item[(h)]  The concatenation $\word{i}(w)\word{i}(v^{-1}w_0)$ is double root free. 
    \item[(i)] The Kazhdan-Lusztig $d$-invariant is $d_{v,w} = \ell(w) - \ell(v)$.
\end{itemize}
\end{theorem}

Note that $S_3$ is the face poset of a bigon, while any other interval of length $2$ is the face lattice of a regular polygon. So the equivalence between (c) and (f) in Theorem \ref{thm:toric-richardson-intro} says that the obvious restriction on length $2$ intervals is enough to guarantee the Bruhat interval is the face lattice of a convex polytope. Conversely, it would be interesting to know for which convex polytopes their face lattice is isomorphic to a Bruhat interval, but we will not pursue that here. We remark that the equivalences (c) $\Longleftrightarrow$ (d) $\Longleftrightarrow$ (f) are contained in unpublished work of Dyer \cite{Dyer_unpublished}. 

We also note that for closed Richardson varieties, the situation from Example~\ref{ex:intro_111} cannot happen, and indeed the equivalence (a) $\Longleftrightarrow$ (b) in Theorem \ref{thm:toric-richardson-intro} shows that they are toric if and only if they are toric with respect to $\dil{v,w}$.

\subsection{Organization of the paper} Section \ref{sec:preliminaries} contains background knowledge. First, it sets up notation for braids and braid words. The section then discusses subword complexes, brick varieties and their stratifications, and braid varieties and their cluster structure. This section contains no new results. In Section \ref{sec:cluster-dilation-braid} we study cluster dilation groups of braid varieties: the main objective of this section is to establish compatibilities of cluster dilation groups with distinct maps between braid varieties. 

Section \ref{sec:torus-actions-bricks} contains our first main result, namely we prove Theorem \ref{thm:main-1-intro}(a) as Theorem \ref{thm:action-extends}. Theorem \ref{thm:intro_properties_of_the_action} is then proved as Lemma \ref{lem:fixed_points} and Corollaries \ref{cor:invariant_divisors} and \ref{cor:proj_moment_to_brick_polytopes}. In Section \ref{sec:toric-brick} we study when the brick variety $\brick(\word{i})$ is $\dil{\word{i}}$-toric, in particular proving Theorems \ref{thm:toric-brick-intro} and \ref{thm:subword-complex-intro} as Theorem \ref{thm:braid-is-torus}. 

In Section \ref{sec:richardsons} we study Richardson varieties and prove Theorem \ref{thm:main-1-intro}(b). In fact, we prove this in the more general case of projected Richardson varieties. We also prove Theorem \ref{thm:richardson-intro} as Theorem \ref{thm:action-extends-richardsons} and Proposition \ref{prop:fixed-points-richardsons},  and Theorem \ref{thm:toric-richardson-intro} as Theorem \ref{thm:toric-richardson-variety}. We finish the paper with several examples in Section \ref{sect:examples}. 

\subsection*{Acknowledgements}

We thank Eugene Gorsky for useful discussions of the work \cite{GKSB}. We also thank Soyeon Kim and Melissa Sherman-Bennett for their comments on an earlier version of this manuscript. MG is very grateful to Paul Philippe, Petra Schwer, and especially Stéphane Gaussent and Zijun Li for many discussions of Bott-Samelson varieties and subword complexes, refuelling his interest in the the latter. He is also thankful to Vincent Pilaud for discussions of the problem of polytopal realizations of subword complexes in early 2010s.

LB acknowledges support from UNAM dgapa 2026 PAPIIT project IN106126 and SECIHTI project CF-2023-G-106. 
MG acknowledges support by the Deutsche Forschungsgemeinschaft (DFG, German Research Foundation) – SFB 1624 – ``Higher structures, moduli spaces and integrability'' – 506632645 and the Procope project ``Buildings, galleries and beyond''.  JS acknowledges support from UNAM PAPIIT Grant IA101526, and SECIHTI project CF-2023-G-106. He is also grateful to the CNRS-UNAM-SECIHTI Solomon Lefschetz International Laboratory for sponsoring his travel to the Institut Fourier in Grenoble, where many ideas of this paper took shape.

\section{Preliminaries}\label{sec:preliminaries} 

We refer the reader to \cite{bb_coxeter} and to \cite{Cox_toric_book, Fulton} for the background on Coxeter groups and on toric geometry and moment maps, respectively. When we talk about roots in the context of Coxeter groups, we always mean the standard geometric representation.

For a concise discussion of relevant notions and properties of cluster varieties in general and of cluster structures on braid varieties, see \cite[Sections 2 and 4]{CGSS} and \cite{Tyler}; see also \cite{FWZ_123, FWZ_6, LamSpeyer-I} for more details.
Since we will study moment maps, all the varieties will be considered over $\CC$.  

\subsection{Braids and words}

Let $(W, S)$ be a Coxeter system with Coxeter-Dynkin diagram $\dynkin$. In a slight abuse of notation, we identify $\dynkin$ with its set of vertices, so the simple reflections are denoted by $s_i$ for $i \in \dynkin$. 
Denote by $\Delta(W)=\{\alpha_i:i\in \dynkin\}$ the simple roots of $W$ and by $\nabla(W)=\{\omega_i:i\in \dynkin\}$ the fundamental weights.
Consider the Artin braid group associated to $W$, denoted by $\Br_W$. It is generated by $\sigma_i$ for $i\in\mathsf{D}$ subject to the relations 
\[
(\sigma_i\sigma_j)^{m_{ij}}=(\sigma_j\sigma_i)^{m_{ji}},
\]
where $m_{ij} = m_{ji}$ is the entry of the Coxeter matrix of $(W, S)$. 
In particular, $\sigma_i\sigma_j=\sigma_j\sigma_i$ if $i$ and $j$ are not connected by an edge in $\dynkin$. The monoid generated by $\sigma_i$'s without their inverses is called the \emph{positive braid monoid} and denoted $\Br_W^+$.

We  consider words in the alphabet $\dynkin$, $\word{i} = i_1i_2\cdots i_r$. 
We always call such words \emph{braid words}. Every braid word defines a positive braid $\beta(\word{i}) = \sigma_{i_1}\cdots \sigma_{i_r}$. 
Two braid words $\word{i}$ and $\word{i}'$ define the same braid if and only if one can be obtained from the other by a finite sequence of  \emph{braid moves} 
\[
\word{i}=\word{i}_1 (ij)^{m_{ij}}\word{i}_2\,\,\,\, \to \,\,\,\, \word{i}_1 (ji)^{m_{ji}}\word{i}_2=\word{i}'.
\]

We have a homomorphism 
\[
\pi:\Br_W^+\to W \quad \text{ given by }\quad \sigma_i\mapsto s_i.
\]
By a slight abuse of notation, if $\word{i}$ is a braid word, we denote $\pi(\word{i}) = \pi(\beta(\word{i}))$. 
For $w\in W$ we denote its minimal length representative in $\Br_W$ by $\beta(w)$. By another abuse of notation, $\word{i}(w)$ 
means an arbitrary word for $\beta(w)$. Note that $\word{i}(w)$ is only defined up to braid moves.

Another important map from the set of braid words to $W$ is the \emph{Demazure product}, defined recursively by
\[
\dem(i)=s_i, \quad \dem(\word{i} j)=\left\{ \begin{matrix}
    \dem(\word{i})s_j & \text{if } \ell(\dem(\word{i})s_j)=\ell(\dem(\word{i}))+1 \\
    \dem(\word{i}) & \text{if } \ell(\dem(\word{i})s_j)=\ell(\dem(\word{i}))-1
\end{matrix}\right..
\]
The Demazure product is well defined and satisfies
\[
\dem(j\word{i})=\left\{ \begin{matrix}
    s_j\dem(\word{i}) & \text{if } \ell(s_j\dem(\word{i}))=\ell(\dem(\word{i}))+1 \\
    \dem(\word{i}) & \text{if } \ell(s_j\dem(\word{i}))=\ell(\dem(\word{i}))-1
\end{matrix}\right..
\]

An equivalent definition of the Demazure product $\delta(\word{i})$ of $\word{i}$ is that it is the largest element of $W$ with respect to the Bruhat order such that $\word{i}$ contains some reduced word for its minimal lift as a (possibly non-consecutive) subword.
Notice that for $k\ge 1$ we have $\delta(i^k)=s_i$ while $\delta(\word{i}(w))=w$ for $w\in W$. 

The Demazure product factors through the positive braid monoid, meaning that if $\beta(\word{i}) = \beta(\word{j})$ then $\dem(\word{i}) = \dem(\word{j})$. Thus, for $\beta \in \Br_W^{+}$, we can define its Demazure product $\dem(\beta) \in W$ unambiguously to be the Demazure product of any word representative for $\beta$.

Finally, for $u,v\in W$ we write 
\[
u*v=\delta(\beta(u)\beta(v)) = \delta(\word{i}(u)\word{i}(v)).
\]

\subsection{Background on subword complexes}
\label{sect:background_subword}
This subsection follows \cite{escobar, cls, PilaudStump-EL,pilaud_stump_15}.

Given a word $\word{i}=i_1\cdots i_r$, a \emph{subword} of $\word{i}$ is a word $\word{j}$ consisting of a subsequence of the letters of $\word{i}$. 
That is, there exists a subsequence $(j_1<\dots<j_s)\subset (1<\dots<r)$ so that $\word{j}=i_{j_1}\cdots i_{j_s}$.
The word $\word{i}$ has $2^r$ subwords. 
Given a subword $\word{j}$, denote by $\word{i}\setminus \word{j}$ the subword of $\word{i}$ corresponding to the subset of indices $[r]\setminus \{j_1,\dots,j_s\}$.
Denote by $\word{i}_{(k)}:=\pi(i_{1}\cdots i_k)$, so that $\word{i}_{(r)}=\pi(\word{i})$ and set $\word{i}_{(0)}=1$. Moreover, given a subword $\word{j}$ of $\word{i}$, denote $\word{j}_{(k)} = \overrightarrow{\prod}_{j_{\ell} \leq k}s_{i_{j_{\ell}}}$.

\begin{definition}
Let $\word{i}=i_1\cdots i_r$ be a word in $S$ and $w \in W$ . The \emph{subword complex} $\Delta(\word{i},w)$ is the simplicial complex on the vertex set $[r]$ whose faces are the subwords $\word{j}$ of $\word{i}$ such that the 
word $\word{i}\setminus\word{j}$ contains a word for $\beta(w)$ as a subword. The facets of  $\Delta(\word{i},w)$ are thus the subwords $\word{j}$ of $\word{i}$ such that the word $\word{i}\setminus\word{j}$ is a word for $\beta(w)$.
\end{definition}

We note that the dual simplicial complex, called  \emph{dual subword complex}, is then the simplicial complex on the vertex set $[r]$ whose faces are the subwords $\word{j}$ of $\word{i}$ which contain  words for $\beta(w)$ as subwords and, in particular, facets are the subwords $\word{j}$ of $\word{i}$ which are words for $\beta(w)$. 

Knutson and Miller \cite{knutson_miller_04} proved that every subword complex is vertex-decomposable, and so shellable. Further, in \cite[Corollary 3.8]{knutson_miller_04} they established that $\Delta(\word{i},w)$, as a topological manifold with boundary, is homeomorphic to a sphere if and only if $\delta(\word{i}) = w$, and is homeomorphic to a ball otherwise. We will be interested only in \emph{spherical} subword complexes, i.e. those with $\delta(\word{i}) = w$.

If two words $\word{i}, \word{i'}$ are related by a commutation move, i.e. a braid move $\word{i}_1 ij \word{i}_2 \to \word{i}_1 ji \word{i}_2$, with $ij = ji$, then for any $w$, the subword complexes $\Delta(\word{i},w)$ and $\Delta(\word{i'},w)$ are canonically isomorphic, cf. \cite[Proposition 3.8]{cls}. For general braid moves, the relation is significantly more complicated and involves edge subdivisions and their inverses, see \cite{G_subword_3, manneville}. 

\begin{example}
\label{ex:A_1_simplex}
For $W = A_1$, with $w_0 = s_1$, consider the word $1^r := 11\ldots 1$ of length $r$. The subword complex $\Delta(1^r, w_0)$ is the boundary complex of an $(r-1)$-dimensional simplex. Indeed, the complement to every proper or empty subword of $1^r$ contains a word $1$ for $\beta(w_0)$, so every proper or empty subword of $1^r$ is a face of $\Delta(1^r, w_0)$.
\end{example}

\begin{example}
For $W = A_2$, with $w_0 = s_1s_2s_1 = s_2s_1s_2$, the subword complex $\Delta(12121, w_0)$ is the boundary complex of a pentagon. Indeed, each pair of cyclically consecutive letters in $12121$ is a complement to a word for $\beta(w_0)$. By doing a braid move, we can go from $12121$ to $12212$. The complex $\Delta(12212, w_0)$ is the boundary complex of a combinatorial square, as in particular the letter at the fourth position is not a vertex. Similarly, we can do a different braid move and go from $12121$ to $11211$, and the complex $\Delta(11211, w_0)$ is again the boundary complex of a combinatorial square, as in particular the letter at the third position is not a vertex.
\end{example}

For a word $\word{i}=i_1\cdots i_r$ define the \emph{root function} 
\begin{equation}\label{eq:root fct}
    \rootf(\cdot,\cdot):\{\text{subwords of }\word{i}\} \times [m] \to \Delta(W), \quad \text{ given by } \rootf(\word{j},k):=(\word{i}\setminus \word{j})_{(k-1)}(\alpha_{i_k}).
\end{equation}
Given a subword complex $\Delta(\word{i},w)$ and a face $\word{j}$, we define the \emph{root configuration} of $\word{j}$ as the multiset 
\begin{equation}\label{eq:root config}
    R(\word{j})=\left\{\{\rootf(\word{j},k):k\in[r]\}\right\}.
\end{equation}

Consider a facet $\word{j}$ of $\Delta(\word{i},w)$ and an index $j$ in $\word{j}$.
If there exists another facet $\word{j}'$ of $\Delta(\word{i},w)$ and an index $j'$ in $\word{j}'$ so that $\word{j}\setminus j=\word{j}'\setminus j'$ we say that $\word{j}$ and $\word{j'}$ are \emph{adjacent facets}.
Moreover, we say that $j$ is \emph{flippable} in $\word{j}$ and that $\word{j}'$ is obtained from $\word{j}$ by \emph{flipping} $j$.

\begin{definition}\label{def:double root free}
    A subword complex $\Delta(\word{i},w)$ is said to have a \emph{double root} if there exists a facet $\word{j}$ and two flippable indices $j\not=j'\in[r]$ in $\word{j}$ for which $\rootf(\word{j},j)=\rootf(\word{j},j')$.
If no double roots exist for $\Delta(\word{i},w)$ it is called \emph{double root free}.
We will say that a word $\word{i}$ is \emph{double root free} if the spherical subword complex $\Delta(\word{i}, \dem(\word{i}))$ is double root free. Note that in the spherical case, each index in each facet is flippable, so the double root condition concerns simply all pairs of indices in facets.
\end{definition}

The ridge graph of a simplicial complex has its maximal faces as vertices and codimension 1 faces as edges. The ridge graph of $\Delta(\word{i},w)$ is also called its \emph{flip graph}. 
It admits a natural acyclic orientation, where a flip replacing a position $i$ by position $j$ is \emph{increasing} if $i < j$. See \cite[Remark 4.5]{knutson_miller_04} and \cite{PilaudStump-EL} for the discussion of properties of the increasing flip graph. Pilaud and Stump show \cite[Proposition 5.1]{PilaudStump-EL} that the word $\word{i}$ is double root free if and only if the increasing flip graph coincides with the Hasse diagram of the partial order on facets given by the transitive closure of increasing flips. This in particular implies nice formulas for the Möbius function of this partial order for double root free words. We note that it follows from the proof of \cite[Proposition 5.1]{PilaudStump-EL} that the word $\word{i}$ is double root free if and only if its flip graph does not contain a triangle.

\begin{example}
Consider the word $1^3 = 111$ from Example~\ref{ex:A_1_simplex}. The subword complex $\Delta(1^3, w_0)$ is the boundary complex of a triangle. For the facet given by the first two positions in $1^3$, both of these indices are flippable and the corresponding roots are both just $\alpha_1$, so the complex has a double root. Note however that for the facet given by the first and the third positions, the roots are $\alpha_1$ and $s_1(\alpha_1) = -\alpha_1$. This shows that it is possible that only some facets of a subword complexes have two copies of the same root.
\end{example}

A spherical subword complex $\Delta(\word{i},w)$ is said to be \emph{root independent} if the roots in the root configuration of each facet are linearly independent. By \cite[Lemma 3.6]{pilaud_stump_15}, this is equivalent to the roots in the root configuration of some facet being linearly independent. We say that a word $\word{i}$ is \emph{root independent} if $\Delta(\word{i},w)$ is root independent. 

Clearly, each root independent word is double root free, but the converse is not true, as the next example shows. Further examples will be discussed in Section~\ref{sect:examples}. 

\begin{example}
\label{ex:112211}
Consider the word $\word{i} = 112211$ in $W = A_2$. For each facet, the root configuration contains exactly one of $\alpha_1$ and $-\alpha_1$, exactly one of $\alpha_1 + \alpha_2$ and $-\alpha_1 -  \alpha_2$, and exactly one of $\alpha_2$ and $-\alpha_2$. There are no repetitions in such configurations, so the word $\word{i}$ is double root free. Since $(\alpha_1, \alpha_1 + \alpha_2, \alpha_2)$ are not linearly independent, $\word{i}$ is not root independent. 
\end{example}

\bigskip

Assume now for the rest of this subsection that $(W, S)$ is a finite Coxeter system. 

For any word $\word{j}$ with $\dem(\word{j}) = \dem$, there exists an isomorphism \cite[Theorem 3.7]{cls}
\begin{equation}
\label{eq:extending_words}
\Delta(\word{j}, \dem) \cong \Delta(\word{j} \word{k}, w_0), \,\, \mbox{for any} \,\, \word{k} \,\, \mbox{of the form} \,\, \word{i}(\dem^{-1}w_0).
\end{equation}
We now recall the definition of certain polytopes introduced in \cite{pilaud_santos_12, pilaud_stump_15} and related to spherical subword complexes. Thanks to isomorphisms \eqref{eq:extending_words}, we may and will assume without loss of generality that $\word{i} = i_1\cdots i_r$ is a braid word with $\dem(\word{i}) = w_0$. Denote by $i_{1^*}$ the unique vertex of $\dynkin$ such that $s_{i_{1}^*} = w_0 s_{i_1} w_0$.
We call the word $\word{i'} = i_2\cdots i_r i_{1}^*$ the \emph{cyclic rotation} of the word $\word{i}$. We have an explicit isomorphism \cite[Proposition 3.9]{cls}
\[
\Delta(\word{i}, w_0) \cong \Delta(\word{i'}, w_0).
\]

For a word $\word{i}$ with $\dem(\word{i}) = w_0$ and a facet $\word{j}$  of the spherical subword complex $\Delta(\word{i}, w_0)$, we define the \emph{weight function}
\begin{equation}\label{eq:weight fct}
    w(\cdot,\cdot):\{\text{subwords of }\word{i}\} \times [m] \to \nabla(W), \quad \text{ given by } w(\word{j},k):=(\word{i}\setminus \word{j})_{(k-1)}(\omega_{i_k})
\end{equation}
and the \emph{brick vector} of the facet $\word{j}$ as
\begin{equation}\label{eq:brick vector}
    B(\word{j})=\sum_{k=1}^r w(\word{j},k).
\end{equation}
The \emph{brick polytope} of $\Delta(\word{i},w_0)$ is the convex hull of all brick vectors
\begin{equation}\label{eq:brick polytope}
    B(\word{i},w_0)=\conv\{B(\word{j}) : \word{j} \in \Delta(\word{i},w_0) \text{ and } (\word{i}\setminus \word{j})_{(m)} = w_0\}.
\end{equation}

Pilaud and Stump \cite[Theorem 4.8, Section 7]{pilaud_stump_15} proved that the polar dual of the brick polytope $B(\word{i},w_0)$ realizes the subword complex  $\Delta(\word{i}, w_0)$ (meaning that the polar dual of $B(\word{i},w_0)$ is a simplicial polytope, isomorphic to $\Delta(\word{i}, w_0)$ as an abstract simplicial complex; equivalently, the nerve complex of $B(\word{i},w_0)$ is isomorphic to $\Delta(\word{i}, w_0)$) if and only if the word $\word{i}$ is root independent. The definition of brick polytopes generalizes to arbitrary words in finite Coxeter systems, cf. \cite{escobar}, and this result generalizes to arbitrary spherical subword complexes for such systems. 

For some root dependent words $\word{i}$, 
it is known that there cannot exist a polytopal realization of $\Delta(\word{i}, w_0)$ whose polar dual would project onto the brick polytope. One example is given by the word $1212121$ in $W = A_2$, cf. \cite[Proposition 5.9]{pilaud_santos_12}.
We will see in Corollary~\ref{cor:polytope_projects_onto_brick} that double root free words behave nicely in this aspect, as for them there always exist  polytopal realizations of $\Delta(\word{i}, w_0)$ whose polar duals project onto the corresponding brick polytopes.

\medskip

\subsection{Brick varieties and their stratifications}

We fix a simple algebraic group $\group$ with Weyl group $W$. More precisely, we 
choose a Borel subgroup $\borel=\borel_+$. Denote by $\borel_-$ its opposite. Further, let $\unipotent_{\pm}$ be the corresponding unipotent subgroups and denote by $T=\borel_+\cap \borel_-$ the algebraic torus, so that $W = N_{\group}(T)/T$. 
Then $\group/\borel$ is called the \emph{flag variety}, it is a smooth projective algebraic variety. The points of $\group/\borel$ are called \emph{flags}. Moreover, the coset of the identity $\borel\in \group/\borel$ is called the \emph{standard flag}.
Points in $\group/\borel$ correspond to Borel subgroups of $\group$. The point corresponding to $\borel_-$ is called the \emph{anti-standard flag}.

\medskip
We denote by $\wo$ the longest element of $W$. By abuse of notation, for an element $w \in W$ we denote by $w \in \group$ a lift to $\group$ so that, for example, $\borel_- = \wo\borel_+\wo$.

The group $\group$ acts diagonally on $\group/\borel \times \group/\borel$, and the orbits are indexed by elements of $W$. More precisely, for every pair of flags $(\flag{1}, \flag{2})$ there exist $g \in G$ and an element $w \in W$ such that $(g\flag{1}, g\flag{2}) = (\borel_{+}, w\borel_{+})$. The element $w$ is uniquely determined by the pair $(\flag{1}, \flag{2})$, and in this case we say that $w$ is the relative position of $(\flag{1}, \flag{2})$ and write $\flag{1} \rel{w} \flag{2}$. Note that $\flag{1} \rel{w} \flag{2}$ is equivalent to $\flag{2} \rel{w^{-1}} \flag{1}$. 

\begin{definition}
    Let $\word{i} = i_1\cdots i_r$ be a braid word. 
    \begin{enumerate}
        \item The \emph{Bott-Samelson variety}  $\BS(\word{i})$ is the variety consisting of all $(r+1)$-tuples of flags $(\flag{0}, \flag{1}, \dots, \flag{r})$ satisfying
        \[
        \borel_{+} = \flag{0} \rel{s_{i_1}} \flag{1} \rel{s_{i_2}} \flag{2} \rel{s_{i_3}} \cdots \rel{s_{i_r}} \flag{r}.
        \]
    \item The \emph{brick variety} $\brick(\word{i})$ is the closed subvariety of $\BS(\word{i})$ cut out by the condition that the last flag $\flag{r}$ is $\dem(\word{i})\borel_{+}$.
    \end{enumerate}

    Note that the Bott-Samelson variety $\BS(\word{i})$ admits a map $\BS({\word{i}}) \to \group/\borel_{+}$ that simply projects an $(r+1)$-tuple to its last flag, and $\brick(\word{i})$ is a fiber of this map. 
\end{definition}

Bott-Samelson varieties are a classical object first studied in \cite{BottSamelson}. The brick variety $\brick(\word{i})$ was introduced by Escobar in \cite[Definition 19]{escobar}. In \cite[Theorem 20]{escobar}, she proves that $\brick(\word{i})$ is a smooth, irreducible projective algebraic variety of dimension $r - \ell(\dem(\word{i}))$.

Given a subword $\word{j}$ of the word $\word{i}$, we define the following locally closed subvariety of the brick manifold $\brick(\word{i})$: 
\[
\brick(\word{i})^{\word{j}} := \{(\flag{0}, \dots, \flag{r}) \in \brick(\word{i}) \mid \flag{t-1} \neq \flag{t} \; \text{if and only if} \; t \in \word{j}\}. 
\]
Note that $\brick(\word{i})^{\word{j}}$ is empty unless $\dem(\word{j}) = \dem(\word{i})$. The following result is due to Escobar.

\begin{theorem}
[\cite{escobar}, Theorem 24]
\label{thm:escobar_stratification}
We have that the decomposition
\[
\brick(\word{i}) = \bigsqcup_{\substack{\word{j} \; \text{a subword of} \; \word{i} \\ \dem(\word{j}) = \dem(\word{i})}} \brick(\word{i})^{\word{j}},
\]
is a stratification of the brick variety whose intersection complex is dual to the subword complex $\Delta(\word{i},\delta(\word{i}))$.
Moreover,
\[
\overline{\brick(\word{i})^{\word{j}}} = \bigsqcup_{\word{k} \; \text{a subword of} \; \word{j}} \brick(\word{i})^{\word{k}}. 
\]
\end{theorem}

Note that if $\dem(\word{i}) = \dem(\word{j})$ then $\overline{\brick(\word{i})^{\word{j}}} \cong \brick(\word{j})$. We will sometimes abuse notation and say that $\brick(\word{j})$ is a closed subvariety of $\brick(\word{i})$. 

The diagonal action of $T$ on the Bott-Samelson variety $\BS(\word{i})$ naturally restricts to the brick variety $\brick(\word{i})$, as the latter is the fiber of the projection to the flag variety. The variety $\BS(\word{i})$ is Hamiltonian symplectic with respect to this action, with an explicit moment map whose description can be found e.g. in \cite{escobar}. (Strictly speaking, one has to choose a $T$-equivariant projective embedding to speak of moment maps; one embedding is described in detail in \cite[Section 4.1]{escobar} and this is the one we use; a different embedding was recently considered in \cite{GU}.) By general results of \cite{Atiyah, GuilleminSternbergI}, the image of the moment map is a convex polytope given by the convex hull of the images of fixed points of the action. The moment map restricts to $\brick(\word{i})$. In \cite[Theorem 21]{escobar}, Escobar proved that the image of  $\brick(\word{i})$ under this restriction is precisely the brick polytope $B(\word{i}, \dem(\word{i}))$. We further have the following.

\begin{theorem}
Let $(W, S)$ be a finite crystallographic Coxeter system. The following are equivalent for a word $\word{i}$:

\begin{enumerate}
\item $\word{i}$ is root independent;
\item $\brick(\word{i})$ is $T$-toric;
\item The polar dual of the brick polytope $B(\word{i},\dem(\word{i}))$ realizes the subword complex  $\Delta(\word{i}, \dem(\word{i}))$. 
\end{enumerate}
\end{theorem}

\begin{proof}
Equivalence (1) $\Longleftrightarrow$ (2) is proved in \cite[Theorem 23]{escobar}. Equivalence (1) $\Longleftrightarrow$ (3) is shown in \cite[Theorem 4.8, Section 7]{pilaud_stump_15}, as discussed in the previous subsection.
\end{proof}

\subsection{Braid varieties and their cluster structure}\label{sec:braid-varieties}

\begin{definition}
    Let $\word{i} = i_1\cdots i_r$ be a braid word. We define the \emph{braid variety} $X(\word{i})$ to be the variety of all $(r+1)$-tuples of flags $(\flag{0}, \flag{1}, \dots, \flag{r})$ satisfying the following conditions
    \[
    \flag{0} = \borel_+,\qquad \flag{t-1} \rel{s_{i_t}} \flag{t} \; \text{for $t = 1, \dots, r$}, \qquad \flag{r} = \dem(\word{i})\borel_+.
    \]
\end{definition}

Note that the braid variety $X(\word{i})$ is an open set inside the brick variety, and in fact $X(\word{i}) \cong \brick(\word{i})^{\word{i}}$. In particular, $X(\word{i})$ is also smooth of dimension $r - \ell(\dem(\word{i}))$. Moreover, it is affine by \cite[Corollary 3.7]{CGGLSS}. Note that, more generally, if $\word{j}$ is a subword of $\word{i}$ and $\dem(\word{j}) = \dem(\word{i})$ then $X(\word{j}) \cong \brick(\word{i})^{\word{j}}$. 

\begin{remark}
Up to a canonical isomorphism, the braid variety $X(\word{i})$ depends only on the positive braid $\beta(\word{i}) \in \Br^{+}$ defined by $\word{i}$. More precisely, if $\word{i}$ and $\word{j}$ differ only by braid moves, then there is a canonical isomorphism $X(\word{i}) \cong X(\word{j})$. In contrast, the brick variety $\brick(\word{i})$ does depend on the chosen braid word, and not just on the underlying braid. 
\end{remark}

Just as with spherical subword complexes, we can work without loss of generality only with words $\word{i}$ with $\dem(\word{i}) = w_0$. This happens because for any word $\word{j}$ with $\dem(\word{j}) = \dem$, there exists an isomorphism 
\[
X(\word{j}) \cong X(\word{j} \word{k}), \,\, \mbox{for any} \,\, \word{k} \,\, \mbox{of the form} \,\, \word{i}(\dem^{-1}w_0).
\]

Assume now that $\word{i} = i_1\cdots i_r$ a braid word with $\dem(\word{i}) = w_0$ and consider its cyclic rotation $\word{i'} = i_2\cdots i_r i_{1}^*$. Again in parallel with the story of spherical subword complexes, there exists an explicit isomorphism  
\[
X(\word{i}) \overset\sim\to X(\word{i'}).
\]

By the main results of \cite{CGGLSS, GLSB, GLSBS}, the braid variety $X(\word{i})$ admits a cluster structure, meaning that the coordinate algebra $\C[X(\word{i})]$ admits the structure of a Fomin-Zelevinsky cluster algebra. By the main result of \cite{CGGSSBS}, the cluster structures constructed in \cite{CGGLSS} and \cite{GLSB} coincide. We quickly review properties of the cluster structure constructed in \cite{CGGLSS}, and refer the reader to \cite{CGGLSS} for details on the construction.

First, we give a presentation of the coordinate algebra $\C[X(\word{i})]$. For this, we fix a \emph{pinning} of the group $\group$: for every vertex $i \in \dynkin$ we fix isomorphisms $x_i: \C \to \unipotent_i^{+}$, $y_i: \C \to \unipotent_i^{-}$, where $\unipotent_i^{+}$ and $\unipotent_i^{-}$ are the corresponding root subgroups of $\group$. These isomorphisms are compatible in the sense that the assignment
\[
\begin{pmatrix}
    1 & z \\ 0 & 1 
\end{pmatrix}
\mapsto x_i(z), \qquad
\begin{pmatrix}
   1 & 0 \\ z & 1  
\end{pmatrix}
\mapsto y_i(z), \qquad
\begin{pmatrix}
 b & 0 \\ 0 & b^{-1}
\end{pmatrix}
\mapsto \chi_i(b)
\]
gives a morphism $\varphi_i: \SL_2 \to \group$, where $\chi_i: \C^{\times} \to T$ is the simple coroot corresponding to $i \in \dynkin$. Any simple algebraic group $\group$ admits such a pinning, and any two pinnings are conjugate, cf. \cite{lusztig-TP1}.

For $z \in \C$ and $i \in \dynkin$, we define the element 
\[
B_i(z) = \varphi_i\begin{pmatrix} z & -1 \\ 1 & 0 \end{pmatrix} \in \group.
\]
Then we have the following affine realization of the braid variety $X(\word{i})$, cf. \cite[Section 3.4]{CGGLSS}. For a braid word $\word{i} = i_1\cdots i_r$:
\begin{equation}\label{eq:affine-realization-braid-variety}
X(\word{i}) \cong \{(z_1, \dots, z_r) \in \C^r \mid \dem(\word{i})^{-1}B_{i_1}(z_1)\cdots B_{i_r}(z_r) \in \borel\} 
\end{equation}
We will call $z_1, \dots, z_r$ the \emph{affine coordinates} on $X(\word{i})$.

The main combinatorial input to define  a seed of a cluster structure on $X(\word{i})$ is that of a \emph{Demazure weave}. We will omit the word ``Demazure'', as these will be the only weaves appearing in the paper. 
We consider a directed graph $\words$ whose vertex set is the set of all braid words. We have two types of arrows. 

\begin{enumerate}
    \item[($\mathrm{A}_1$)] There is an arrow $\word{i}_1(ij\dots)\word{i}_2 \to \word{i}_1(ji\dots)\word{i}_2$ encoding braid moves. Note that there is also an arrow in the opposite direction.
    \item[$(\mathrm{A}_2$)] There is an arrow $\word{i}_1(ii)\word{i}_2 \to \word{i}_1(i)\word{i}_2$. There is no arrow going in the opposite direction. 
\end{enumerate}

We will call an arrow of the form ($\mathrm{A}_2$) \emph{frozen} if $\dem(\word{i}_1\word{i}_2) < \dem(\word{i}_1i\word{i}_2)$. 
Note that the Demazure product gives a bijection between $W$ and the set of connected components of the directed graph $\words$. A \emph{weave} $\weave: \word{i} \to \word{j}$ is a directed path in the graph $\words$. By definition, a weave can be decomposed into arrows of the form ($\mathrm{A}_1$) and ($\mathrm{A}_2$) above. 
These encode maps between braid varieties, as follows. 

\begin{enumerate}
    \item[($\mathrm{A}_1$)]  These arrows encode isomorphisms coming from braid moves:
    \begin{equation}\label{eq:2m-valent}
    X(\word{j}) \xrightarrow{\cong} X(\word{i}), \qquad \text{where} \; \word{i} =  \word{i}_1\!(ij\dots)\word{i}_2, \text{ and } \word{j} = \word{i}_1(ji\dots)\word{i}_2 
    \end{equation}
    \item[($\mathrm{A}_2$)]  These arrows encode an inclusion
    \begin{equation}\label{eqa:3-valent}
    X(\word{j}) \times \C^{\times} \hookrightarrow X(\word{i}), \,\,\,\text{where} \; \word{i} = \word{i}_1 (i i) \word{i}_2 \; \text{ and } \; \word{j} = \word{i}_1 (i ) \word{i}_2
    \end{equation}
    that identifies $X(\word{j}) \times \C^{\times}$ with a principal open set in $X(\word{i})$.
\end{enumerate} 

A \emph{complete Demazure weave} is a weave $\weave: \word{i} \to \dem(\word{i})$, where we abuse the notation and the right-hand side stands for any reduced word for $\dem(\word{i})$.

\begin{theorem}[Theorem 1.1, Lemma 7.2, Theorem 5.12 and Theorem 5.31 in \cite{CGGLSS}]\label{thm:cluster-structure}
The algebra $\C[X(\word{i})]$ admits a cluster structure. Moreover:
\begin{enumerate}
\item Every complete Demazure weave $\weave: \word{i} \to \dem(\word{i})$ defines a seed in this cluster structure. Cluster variables in this seed are in bijection with arrows of the form ($\mathrm{A}_2$) in the path $\weave: \word{i} \to \dem(\word{i})$. Frozen variables correspond to frozen arrows. \cite[Theorem 1.1, Lemma 7.2]{CGGLSS}
\item Any affine coordinate $z_1, \dots, z_r$ is either identically vanishing on $X(\word{i})$ or a cluster monomial. \cite[Lemma 5.24, Corollary 5.34]{CGGLSS}
\item Assume that $\dem(\word{i}) = w_0$. Let $\word{i}' = i_2\cdots i_{r}i_{1}^{\ast}$. Then there is a cluster quasi-isomorphism $X(\word{i}) \to X(\word{i}')$, called the \emph{cyclic rotation}. \cite[Theorem 5.31]{CGGLSS}.
\end{enumerate}
\end{theorem}

Finally, we mention a graphical way of depicting weaves, as in \cite{CGGS1, CGGLSS}, adapting a symplectic-geometrically motivated construction of \cite{CasalsZaslow} to the Lie theoretic context. A braid word $\word{i} = i_1\dots i_r$ is encoded by $r$ vertical lines, colored left-to-right by the letters $i_1, \dots, i_r$. Arrows of the form ($\mathrm{A}_1$) are encoded by the following local moves, depending on the value of $m_{ij}$, which are read top-to-bottom:

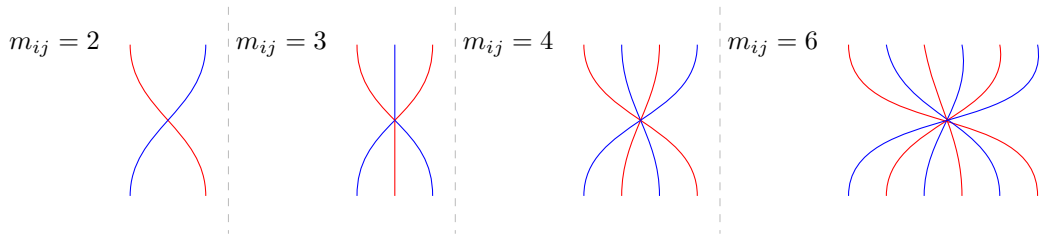
\begin{center}
    \begin{tikzpicture}
    \draw[color=red] (0,2) to[out=270, in=90] (1,0);
    \draw[color=blue] (1,2) to[out=270, in=90] (0,0);
    \node at (-1, 2) {\small{$m_{ij} = 2$}};

\draw[dashed, color=lightgray] (1.3, 2.5) to (1.3,-0.5);
    \draw[color=red] (3,2) to[out=270, in=135] (3.5, 1);
    \draw[color=blue] (3.5, 2) to (3.5,1);
    \draw[color=red] (4, 2) to[out=270, in=45] (3.5,1);
    \draw[color=blue] (3.5, 1) to[out=225, in=90] (3, 0);
    \draw[color=red] (3.5, 1) to (3.5, 0);
    \draw[color=blue] (3.5, 1) to[out=315, in=90] (4, 0);
    \node at (2,2) {\small{$m_{ij} = 3$}};

\draw[dashed, color=lightgray] (4.3, 2.5) to (4.3, -0.5);
\draw[color=red] (6,2) to[out=270, in=90] (7.5, 0);
\draw[color=blue] (6.5, 2) to[out=270, in=90] (7, 0);
\draw[color=red] (7, 2) to[out=270, in=90] (6.5, 0);
\draw[color=blue] (7.5, 2) to[out=270, in=90] (6,0);
\node at (5, 2) {\small{$m_{ij} = 4$}};

\draw[dashed, color=lightgray] (7.8, 2.5) to (7.8, -0.5);
\draw[color=red] (9.5, 2) to[out=270, in=90] (12,0);
\draw[color=blue] (10, 2) to[out=280, in=90] (11.5, 0);
\draw[color=red] (10.5, 2) to[out=280, in=90] (11, 0);
\draw[color=blue] (11,2) to[out=280, in=90] (10.5, 0);
\draw[color=red] (11.5, 2) to[out=280, in=90] (10, 0);
\draw[color=blue] (12,2) to[out=280, in=90] (9.5, 0);
\node at (8.5, 2) {\small{$m_{ij} = 6$}};
    \end{tikzpicture}
\end{center}
On the other hand, arrows of the form ($\mathrm{A}_2$) are encoded by the following local move:
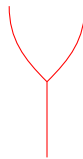
\begin{center}
\begin{tikzpicture}
        \draw[color=red] (3,2) to[out=270, in=135] (3.5, 1);
    \draw[color=red] (4, 2) to[out=270, in=45] (3.5,1);
    \draw[color=red] (3.5, 1) to (3.5, 0);
\end{tikzpicture}
\end{center}
For this reason we will sometimes refer to arrows of type ($\mathrm{A}_2$) as \emph{trivalent vertices} of a weave.

\section{The cluster dilation group of braid varieties}\label{sec:cluster-dilation-braid}

\subsection{Cluster dilation groups} Let us recall that, if $A$ is a cluster algebra, an algebra automorphism $\varphi: A \to A$ is called a \emph{cluster dilation} if for every cluster variable $z \in A$ there exists a nonzero scalar $\lambda_z \in \C^{\times}$ such that $\varphi(z) = \lambda_z z$. The set of all cluster dilations forms a group, the \emph{cluster dilation group} of $A$. It has an explicit description, as follows. Let us fix a seed $\seed = (\widetilde{B}, \widetilde{\mathbf{x}})$ of $A$. The $(n+m) \times n$ extended exchange matrix $\widetilde{B}$ defines a map:
\[
\mathsf{mult}(\widetilde{B}^{T}): (\C^{\times})^{n+m} \to (\C^{\times})^n
\]
and the cluster dilation group is isomorphic to $\ker(\mathsf{mult}(\widetilde{B}^{T}))$. In particular, if $\widetilde{B}$ has \emph{really full rank}, that is, if the $\Z$-span of the rows of the matrix $\widetilde{B}$ is the entire $\Z^n$,  then the cluster dilation group is a torus of rank $m$, the number of frozen variables of $A$.

We will be interested in the case where $A = \C[X(\word{i})]$, the coordinate ring of the braid variety $X(\word{i})$, with its cluster structure from Section \ref{sec:braid-varieties}. We will denote the cluster dilation group by $\dil{\word{i}}$. It is a torus of rank equal to the number of frozen variables, that we denote by $\frozen{\word{i}}$ \footnote{When $\word{i}$ is a word associated to an open Richardson variety $\orich{v,w}$, $\frozen{\word{i}}$ is known to be the Kazhdan-Lusztig $d$-invariant of the Bruhat interval $[v, w]$, cf. Proposition \ref{prop:KL-d-invariant} below.}

Recall that the variety $X(\word{i})$ comes equipped with affine coordinates $z_1, \dots, z_r$, where $r$ is the length of $\word{i}$. By Theorem~\ref{thm:cluster-structure}.(2), each coordinate $z_t$ is either zero or a cluster monomial. First we recall the characterization of those indices $t$  for which $z_t$ is identically zero. 

\begin{definition}
    Let $\word{i} = i_1\cdots i_r$ be a braid word. We say that $t$ is an {\it essential crossing} of $\word{i}$ if $\dem(i_1\cdots {i_{t-1}i_{t+1}\cdots i_r}) < \dem(\word{i})$. 
\end{definition}

\begin{lemma}
    The following are equivalent:
    \begin{enumerate}
    \item The crossing $t$ is essential in $\word{i}$.
    \item The variable $z_t$ is identically zero on $X(\word{i})$. 
    \end{enumerate}
\end{lemma}
\begin{proof}
   Up to cyclic rotation we may assume $t = 1$, and the result follows from \cite[Lemma 5.24.(1)]{CGGLSS}.
\end{proof}

We denote by $\noness{\word{i}} := \{t \in [r] \mid t \; \text{is not an essential crossing of} \; \word{i}\}$. For every $t \in I_{\word{i}}$, the variable $z_t$ is a nonzero cluster monomial, so the torus $\dil{\word{i}}$ acts on $z_t$ by a character $\zeta_t: \dil{\word{i}} \to  \C^{\times}$. We arrange these in a map
\begin{equation}\label{eq:zeta-embedding-1}
\zeta: \dil{\word{i}} \to (\C^{\times})^{\noness{\word{i}}}.
\end{equation}
By definition, the group $\dil{\word{i}}$ acts faithfully on $\C[X(\word{i})]$, so $\zeta$ is an injective map. In other words, $\dil{\word{i}}$ is the subgroup of $(\C^{\times})^{\noness{\word{i}}}$ consisting of those elements that make every cluster variable (thought of as a polynomial in the variables $z_1, \dots, z_r$) homogeneous. We can embed $(\C^{\times})^{\noness{\word{i}}} \hookrightarrow (\C^{\times})^{r}$ by making the coordinates associated to essential crossings equal to $1$. We will sometimes need the composition
\begin{equation}\label{eq:zeta-embedding-2}
\dil{\word{i}} \xrightarrow{\zeta} (\C^{\times})^{\noness{\word{i}}} \hookrightarrow (\C^{\times})^{r}.
\end{equation}

\subsection{Weaves and the cluster dilation group} In this section, we study the interaction between the cluster dilation group of braid varieties and the calculus of weaves.  Our goal is to prove that a weave $\weave: \word{i} \to \word{j}$ induces a map of cluster dilation groups $\dil{\word{i}} \to \dil{\word{j}}$, where $\word{i}$ and $\word{j}$ are braid words.

\begin{lemma}\label{lem:easy-2m-valent}
The isomorphism \eqref{eq:2m-valent} induces an isomorphism of cluster dilation groups $\dil{\word{i}} \cong \dil{\word{j}}$.
\end{lemma}
\begin{proof}
    Let us denote by $z_1, \dots, z_r$ the affine coordinates on $X(\word{i})$ and by $w_1, \dots, w_r$ the affine coordinates on $X(\word{j})$, so that the pullback of the isomorphism \eqref{eq:2m-valent} is the algebra isomorphism $\varphi^{\ast}: \C[X(\word{i})] \to \C[X(\word{j})]$, $\varphi^{\ast}(z_i) = f_i(w_1, \dots, w_r)$. We claim that the elements $f_i(w_1, \dots, w_r)$ are homogeneous with respect to the $\dil{\word{j}}$-action on $\C[X(\word{j})]$. By \cite[Sects. 4.4, 4.5]{GLSB}, see also \cite[Proposition 4.46]{CGGLSS}, the isomorphism $\varphi^{\ast}$ sends cluster monomials to cluster monomials. So $f_i(w_1, \dots w_r) \in \C[X(\word{j})]$ is in fact a cluster monomial, and is thus homogeneous with respect to the $\dil{\word{j}}$-action. We denote by $\wt(f_t)$ the $\dil{\word{j}}$-weight of $f_t$, and make the convention that $\wt(f_t) = 0$ if $f_t = 0$. 

    Let us define a morphism $\Phi: \dil{\word{j}} \to (\C^{\times})^{r}$ by $\xi \mapsto (\xi^{\wt(f_1)}, \dots, \xi^{\wt(f_r)})$. We claim that its image is contained in $\dil{\word{i}}$, where we embed $\dil{\word{i}}$ in $(\C^{\times})^{r}$ via \eqref{eq:zeta-embedding-2}. First, note that if $t$ is an essential crossing of $\word{i}$ then $z_t$, and thus $f_t$, are zero, so $\xi^{\wt(f_i)} = 1$ and the image of $\dil{\word{j}}$ is contained in (the image of) $(\C^{\times})^{\noness{\word{i}}}$, as needed. Now we need to show that if $x \in \C[X(\word{i})]$ is a cluster variable then an element of the form $\Phi(\xi)$ acts homogeneously on it. But if $x$ is a cluster variable then $\varphi^{\ast}(x) \in \C[X(\word{j})]$ is also a cluster variable and it is easy to verify that $\varphi^{\ast}(\Phi(\xi)x) = \xi^{\wt(\varphi^{\ast}(x))}\varphi^{\ast}(x)$. Since $\varphi^{\ast}$ is an isomorphism, $\Phi(\xi)x = \xi^{\wt^{\varphi^{\ast}(x)}}x$ and it follows that $\Phi(\dil{\word{j}})$ is indeed contained in $\dil{\word{i}}$. That $\Phi: \dil{\word{j}} \to \dil{\word{i}}$ is an isomorphism follows by considering the inverse map $\C[X(\word{i})] \to \C[X(\word{j})]$.  
\end{proof}

\begin{lemma}\label{lem:3valent-dilation}
The map \eqref{eqa:3-valent} induces a map of cluster dilation groups $\dil{\word{i}} \to \dil{\word{j}}$.
\end{lemma}
\begin{proof}
Let $z_1, \dots, z_{r+1}$ be the affine coordinates on $X(\word{i})$ and $w_1, \dots, w_r$ the affine coordinates on $X(\word{j})$. Moreover, let $t_0$ be the index corresponding to the second $i$ in the middle $ii$ of $\word{i}$. Then just as in the proof of Lemma \ref{lem:easy-2m-valent} we have an isomorphism
\[
\varphi^{\ast}: \C[X(\word{j})][w^{\pm 1}] \xrightarrow{\cong} \C[X(\word{i})][z_{t_0}^{-1}],
\]
for a new variable $w$ corresponding to the $\CC^\times$ factor on the left hand side of  \eqref{eqa:3-valent}.
There exists a seed of $\C[X(\word{i})]$ on which the variable $z_{t_0}$ is a cluster variable (take, for example, any weave whose top trivalent vertex corresponds to this move $ii \to i$, cf. \cite[Theorem 5.19]{CGGLSS}), so the algebra $\C[X(\word{i})][z_{t_0}^{-1}]$ admits a cluster structure given by freezing the $z_{t_0}$-variable, see \cite[Theorem 5.35]{CGGLSS} and \cite[Lemma 3.4]{muller-locally-acyclic}. Note that every cluster monomial of $\C[X(\word{i})][z_{t_0}^{-1}]$ is a cluster monomial of $\C[X(\word{i})]$ times a (possibly negative) power of $z_{t_0}$. In particular, every cluster monomial in $\C[X(\word{i})][z_{t_0}^{-1}]$ is homogeneous with respect to the $\dil{\word{i}}$-action.

The algebra $\C[X(\word{j})][w^{\pm 1}]$ also admits a cluster structure obtained by appending a disjoint frozen variable to the cluster structure on $\C[X(\word{j})]$. The map $\varphi^{\ast}$ is a cluster quasi-isomorphism, see \cite[Lemma 4.19]{CGSS}, so it sends cluster monomials to cluster monomials. Now the proof continues as in that of Lemma \ref{lem:easy-2m-valent} above.
\end{proof}

    \subsection{Fixed points} We study the fixed points of the braid variety $X(\word{i})$ under the action of the cluster dilation group. Recall the maximal torus $T = \borel_+ \cap \borel_{-}$. Since the torus $T$ fixes both the standard $\borel_{+}$ and the antistandard $\borel_{-}$ flags, it acts on the braid variety $X(\word{i})$ diagonally. This action is by cluster dilations \cite[Corollary 4.29]{CGSS}\footnote{The statement in \cite{CGSS} is only for type $A$, but the same proof works in an arbitrary type.}, so we have a map $T \to \dil{\word{i}}$.

\begin{lemma}\label{lem:fixed-points-braid} 
We have that the $T$-fixed points and the $\dil{\word{i}}$-fixed points on $X(\word{i})$ coincide, that is,
\[
X(\word{i})^{T} = X(\word{i})^{\dil{\word{i}}}
\]
Moreover in terms of the affine coordinates on $X(\word{i})$:
\[
X(\word{i})^{T} = X(\word{i})^{\dil{\word{i}}} = \begin{cases} \{(0, 0, \dots, 0)\} & \text{if} \: \pi(\word{i}) = \dem(\word{i}), \\
\emptyset & \text{else}. \end{cases} 
\]
\end{lemma}
\begin{proof}
Since $T$ acts by cluster dilations, we have $X(\word{i})^{\dil{\word{i}}} \subseteq X(\word{i})^{T}$. In terms of flags, a point fixed under $T$ must consist entirely of coordinate flags. This forces the affine coordinates of such a point to be $(0,0, \dots, 0)$, and it is straightforward to check that this point in fact belongs to the braid variety if and only if $\pi(\word{i}) = \dem(\word{i})$. Since the affine coordinates are $\dil{\word{i}}$-homogeneous, this point is also fixed under the $\dil{\word{i}}$-action. 
\end{proof}

\section{Torus actions on brick varieties}\label{sec:torus-actions-bricks}
Our goal is to show that the action of $\dil{\word{i}}$ on $X(\word{i})$ admits a natural extension to $\brick(\word{i})$. Let us summarize the strategy of the proof. 
\begin{enumerate}
    \item First, we will study a cover of $\brick(\word{i})$ by natural affine open charts. This is well-known, see e.g. \cite{cheah}, but we will need the details of the construction, so we will review it carefully.
    \item One of the affine charts in the previous point is precisely the braid variety $X(\word{i})$. We will use the explicit form of the transition functions between charts to give a compatible $\dil{\word{i}}$-action (= grading on the coordinate algebra by the character lattice $\charlat(\dil{\word{i}})$) on each affine chart. 
\end{enumerate}
Part (2) is the key technical step. We need to show that, under the transition functions between affine charts, coordinates in each open chart are $\dil{\word{i}}$-homogeneous rational functions in the coordinates $z_1, \dots, z_r$ of the braid variety $X(\word{i})$, and use this to define a $\charlat(\dil{\word{i}})$-grading on the coordinate algebra of the chart.

\subsection{Affine charts} Let us study the open affine charts of $\brick(\word{i})$ and their transition functions. The index set of these charts will be a subset of
\begin{equation}\label{eq:index-set-charts}
2^{\word{i}} := \{\phi: [r] \to \{+, -\}\}.
\end{equation}
In fact, $2^{\word{i}}$ is the index set of a natural set of charts on the Bott-Samelson variety $\BS(\word{i})$, of which $\brick(\word{i})$ is a closed subvariety, and the charts on $\brick(\word{i})$ are simply intersections of these Bott-Samelson charts with the brick variety. Let us define the Bott-Samelson charts. For this, we define the elements
\[
B_i^{+}(z) := \varphi_i\begin{pmatrix} z & -1 \\ 1 & 0 \end{pmatrix}, \qquad B_i^{-}(z) := \varphi_i\begin{pmatrix}1 & 0 \\ z & 1\end{pmatrix},
\]

\begin{definition}
For $\phi \in 2^{\word{i}}$, define the chart $C(\phi) \subseteq \BS(\word{i})$ via the isomorphism
\[
\C^{r} \to C(\phi)
\]
that sends a tuple $(z_{1, \phi}, \dots z_{r, \phi})$ to the tuple of flags $(\flag{0}, \dots, \flag{r})$ so that the $t$-th flag is given by 
\[
B_{i_1}^{\phi(1)}(z_{1, \phi})B_{i_2}^{\phi(2)}(z_{2, \phi})\cdots B_{i_t}^{\phi(t)}(z_{t, \phi})\borel,
\]
and we denote $C_{\word{i}}(\phi) := C(\phi) \cap \brick(\word{i})$. 
\end{definition}

Note that $\BS(\word{i}) = \bigcup C(\phi)$ and $C(\phi)$ is open in $\BS(\word{i})$. By definition, $C(\phi)$ is an affine space, and each $C_{\word{i}}(\phi)$, when nonempty, is the affine closed subvariety of $C_{\word{i}}$ given by the equation
\[
\dem(\word{i})^{-1}B_{i_1}^{\phi(1)}(z_{1, \phi})\cdots B_{i_r}^{\phi(r)}(z_{r, \phi}) \in \borel.
\]

\begin{proposition}\label{prop:irreducible}
    For any $\phi \in 2^{\word{i}}$, the variety $C_{\word{i}}(\phi)$ is either empty or it is a smooth, irreducible affine variety of dimention $r - \ell(\dem(\word{i}))$. 
\end{proposition}
\begin{proof}
    If $C_{\word{i}}(\phi)$ is nonempty, then it is open in the smooth, irreducible variety $\brick(\word{i})$, so the result follows.
\end{proof}

\begin{remark}\label{rmk:equations-braid}
    Let us denote by $\phi_{+} \in 2^{\word{i}}$ the constant function so that $\phi_{+}(t) = +$ for every $t \in [r]$. By definition, $C_{\word{i}}(\phi_{+}) = X(\word{i})$, the braid variety. More generally, let $\word{j}$ be a subword of $\word{i}$ with the same Demazure product. Define $\phi_{\word{j}} \in 2^{\word{i}}$ by
    \[
    \phi_{\word{j}}(t) = \begin{cases} +, & \text{if} \; t \in \word{j} \\ -, & \text{else}. \end{cases}
    \]
    Then the braid variety $X(\word{j})$ is identified with the closed subvariety of $C_{\word{i}}(\phi_{\word{j}})$ given by $z_{t, \phi_{\word{j}}} = 0$ for $t \not \in \word{j}$. 
\end{remark}

\begin{example}\label{ex:running-1}
Let us take $\word{i} = 111$, so that $\BS(\word{i}) = (\PP^1)^3$ and $\brick(\word{i}) = (\PP^1)^2$. We denote $\phi \in 2^{\word{i}}$ by $\phi(1)\phi(2)\phi(3)$. We have eight charts in $\BS(\word{i})$:
\[
\begin{array}{rl}
C(+++) = & \{([z_{1}:1], [z_1z_2-1:z_2], [z_1z_2z_3-z_1-z_3: z_2z_3-1]) \mid z_1, z_2, z_3 \in \C\},  \\
C(++-) = & \{([z_{1}:1], [z_1z_2-1:z_2], [{z_1z_2-z_1z_3-1}: z_2-z_3]) \mid z_1, z_2, z_3 \in \C\}, \\
C(+-+) = & \{([z_1:1], [z_1-z_2:1], [z_1z_3-z_2z_3-1:z_3]) \mid z_1, z_2, z_3 \in \C\}, \\
C(+--) = & \{([z_1:1], [z_1-z_2:1], [z_1-z_2-z_3:1]) \mid z_1, z_2, z_3 \in \C\}, \\
C(-++) = & \{([1:z_1], [z_2:z_1z_2+1], [z_2z_3-1:z_1z_2z_3 - z_1 + z_3]) \mid z_1, z_2, z_3 \in \C\}, \\
C(-+-) = & \{([1:z_1], [z_2:z_1z_2+1], [z_2-z_3:z_1z_2-z_1z_3+1]) \mid z_1, z_2, z_3 \in \C\}, \\
C(--+) = & \{([1:z_1], [1:z_1+z_2], [z_3: z_1z_3 + z_2z_3+1]) \mid z_1, z_2, z_3 \in \C\}, \\
C(---) = & \{([1:z_1], [1:z_1+z_2], [1:z_1+z_2+z_3]) \mid z_1, z_2, z_3 \in \C\}. 
\end{array}
\]
Note that $C_{\word{i}}(---) = \emptyset$, while $C_{\word{i}}(\phi)$ is a smooth affine variety for any other function $\phi \in 2^{\word{i}}$. 
\end{example}

As we have seen in Example \ref{ex:running-1}, $\brick(\word{i})$ is covered by $C_{\word{i}}(\phi)$ where $\phi$ runs in a, in general, proper subset of $2^{\word{i}}$, for it may happen that $C_{\word{i}}(\phi) = \emptyset$. We can be more precise about the collections of functions that we can take to cover $\brick(\word{i})$.

\begin{definition}
    Let $\phi \in 2^{\word{i}}$. The support of $\phi$ is $\supp(\phi) := \{t \in [r] \mid \phi(t) = +\}$.
\end{definition}

We may and will consider $\supp(\phi)$ as a subword of $\word{i}$. 

\begin{lemma}
We have $\brick(\word{i}) = \bigcup C_{\word{i}}(\phi)$, where $\phi$ runs over the set of functions with $\dem(\supp(\phi)) = \dem(\word{i})$.
\end{lemma}
\begin{proof}
Let $\flag{\bullet} = (\flag{0}, \dots, \flag{r}) \in \brick(\word{i})$. Note that the set of elements $t \in [r]$ such that $\flag{t-1} \neq \flag{t}$ must form a subword of $\word{i}$ with Demazure product equal to $\dem(\word{i})$. But this means that $\flag{\bullet}$ is contained in a set $C_{\word{i}}(\phi)$ with $\dem(\supp(\phi)) = \dem(\word{i})$. 
\end{proof}

Note that an essential crossing belongs to any subword of $\word{i}$ with the same Demazure product as $\word{i}$, so we have the following result.

\begin{corollary}\label{cor:not-essential}
    We have $\brick(\word{i}) = \bigcup C_{\word{i}}(\phi)$, where $\phi$ runs over the set of functions such that $\phi(j) = +$ for every essential crossing of $\word{i}$.
\end{corollary}

\subsection{Transition functions} Let $\phi_1, \phi_2 \in 2^{\word{i}}$. The purpose of this section is to describe the transition function between the charts $C_{\word{i}}(\phi_1)$ and $C_{\word{i}}(\phi_2)$. Let us first assume that $\phi_1$ and $\phi_2$ differ in a unique element $t_0 \in [r]$ and, moreover, that $\phi_1(t_0) = +$ and $\phi_2(t_0) = -$, while $\phi_1(t) = \phi_2(t) = +$ for every $t > t_0$. Note that the intersection $C_{\word{i}}(\phi_1)\cap C_{\word{i}}(\phi_2)$ is given by $z_{t_0, \phi_1} \neq 0$. Note that we have the following identity of elements of $\SL_2$:
\[
\begin{pmatrix}
z & -1 \\ 1 & 0 
\end{pmatrix} =
\begin{pmatrix}
1 & 0 \\
z^{-1} & 1 
\end{pmatrix}
\begin{pmatrix}
z & -1 \\
0 & z^{-1}
\end{pmatrix}
\]
Let us define $U_i(z) = \varphi_i\begin{pmatrix} z & -1 \\ 0 & z^{-1}\end{pmatrix} \in \borel$, where $\varphi: \SL_2 \to \group$ comes from a pinning, cf. Section \ref{sec:braid-varieties}. By \cite[Corollary 3.9]{CGGLSS}, we can \lq\lq slide\rq\rq \, the element $U_{j_0}(z_{t_0, \phi_1})$ to the right, giving us an equality
\begin{equation}\label{eq:transition-function}
\begin{array}{rl}
B_{i_1}^{\phi_1(1)}(z_{1, \phi_1})\cdots B_{i_t}^{\phi_1(t)}(z_{t, \phi_1}) & = \\ & B_{i_1}^{\phi_1}(z_{1, \phi_1})\cdots B_{i_{t_0}}^{\phi_2(t_0)}(z_{t_0, \phi_1}^{-1})B_{i_{t_0+1}}^{\phi_2(t_0+1)}(z_{t_0+1, \phi_1}')\cdots B^{\phi_2(t)}_{i_{t}}(z_{t, \phi_1}')U_{t}
\end{array}
\end{equation}
for every $t \geq t_0$, where $U_t \in \borel$ and $z'_{t, \phi_1}$ is a polynomial function in $z_{t_0, \phi_1}^{\pm 1}, z_{t_0+1, \phi_1}, \dots, z_{t, \phi_1}$ for every $t \geq t_0$. These give precisely the transition functions between $C_{\word{i}}(\phi_1)$ and $C_{\word{i}}(\phi_2)$. A general transition function is a composition of these and their inverses.

\begin{remark}\label{rmk:sliding-inverse}
    Let us elaborate on the inverse of the transition function, as this will be needed later. Note that
    \[
    \begin{pmatrix} z & -1 \\ 0 & z^{-1} \end{pmatrix}^{-1} = \begin{pmatrix} z^{-1} & 1 \\ 0 & z\end{pmatrix} = -\begin{pmatrix} -z^{-1} & -1 \\ 0 & -z\end{pmatrix}
    \]
    Thus, up to signs and recalling that $z_{t_0, \phi_2} = z_{t_0, \phi_1}^{-1}$, the inverse of the transition function is given by sliding the element $U(-z_{t_0, \phi_2})$ to the right. 
\end{remark}

\begin{example}\label{ex:running-transitions}
Let us find the transition function between $C_{\word{i}}(+++)$ and $C_{\word{i}}(--+)$ from Example \ref{ex:running-1}. We factor this into the composition of two transition functions
\[
C_{\word{i}}(+++) \dashrightarrow C_{\word{i}}(-++) \dashrightarrow C_{\word{i}}(--+).
\]
For the first transition function, assuming $z_1 \neq 0$ we have
\[
\begin{array}{rl}
\begin{pmatrix} z_1 & - 1 \\ 1 & 0 \end{pmatrix}\begin{pmatrix} z_2 & - 1 \\ 1 & 0 \end{pmatrix}\begin{pmatrix} z_3 & - 1 \\ 1 & 0 \end{pmatrix} = & 
\begin{pmatrix} 1 & 0 \\ z_1^{-1} & 1 \end{pmatrix} \begin{pmatrix} z_1 & - 1 \\ 0 & z_1^{-1} \end{pmatrix} \begin{pmatrix} z_2 & - 1 \\ 1 & 0 \end{pmatrix}\begin{pmatrix} z_3 & - 1 \\ 1 & 0 \end{pmatrix} \\
= & \begin{pmatrix} 1 & 0 \\ z_1^{-1} & 1 \end{pmatrix}  \begin{pmatrix} z_1^{2}z_2 - z_1 & - 1 \\ 1 & 0 \end{pmatrix} \begin{pmatrix} z_1^{-1} & 0 \\ 0 & z_1 \end{pmatrix} \begin{pmatrix} z_3 & - 1 \\ 1 & 0 \end{pmatrix} \\
= & \begin{pmatrix} 1 & 0 \\ z_1^{-1} & 1 \end{pmatrix}  \begin{pmatrix} z_1^{2}z_2 - z_1 & - 1 \\ 1 & 0 \end{pmatrix} \begin{pmatrix} z_1^{-2}z_3 & -1 \\ 1 & 0 \end{pmatrix} \begin{pmatrix} z_1 & 0 \\ 0 & z_1^{-1} \end{pmatrix}.
\end{array}
\]
So the transition map is given by $(z_1, z_2, z_3) \mapsto (z_1^{-1}, z_1^{2}z_2 - z_1, z_1^{-2}z_3)$. For the second transition map, assuming $z_2 \neq 0$ we have
\[
\begin{array}{rl}
\begin{pmatrix} 1 & 0 \\ z_1 & 1\end{pmatrix} \begin{pmatrix} z_2 & -1 \\ 1 & 0\end{pmatrix} \begin{pmatrix} z_3 & -1 \\ 1 & 0\end{pmatrix} = \begin{pmatrix} 1 & 0 \\ z_1 & 1\end{pmatrix} 
\begin{pmatrix} 1 & 0 \\ z_2^{-1} & 1\end{pmatrix}\begin{pmatrix} z_2^2z_3 - z_2 & -1 \\ 1 & 0\end{pmatrix}\begin{pmatrix} z_2^{-1} & 0 \\ 0 & z_2\end{pmatrix}.
\end{array} 
\]
so this transition function is given by $(z_1, z_2, z_3) \mapsto (z_1, z_2^{-1}, z_2^2z_3 - z_2)$. Thus, the transition map $C_{\word{i}}(+++) \to C_{\word{i}}(--+)$ is given by
\[
(z_1, z_2, z_3) \mapsto (z_1^{-1}, (z_1^{2}z_2 - z_1)^{-1}, (z_1^{2}z_2 - z_1)^2z_1^{-2}z_3 - (z_1^2z_2 - z_1)). 
\]

Note that, taking pullbacks, the map 
\[
\begin{array}[t]{rcl}
\C[C_{\word{i}}(--+)][z_{1,--+}^{-1}, z_{2,--+}^{-1}] &\to & \C[C_{\word{i}}(+++)][z_{1,+++}^{-1}, (z_{1,+++}^{2}z_{2,+++} - z_{1,+++})^{-1}] \\
z_{1,--+} & \mapsto  & z_{1,+++}^{-1} \\
z_{2, --+} & \mapsto & (z_{1,+++}z_{2,+++} - z_{1,+++})^{-1} \\
z_{3,--+} & \mapsto &  (z_{1,+++}^{2}z_{2,+++} - z_{1,+++})^2z_{1,+++}^{-2}z_{3,+++} - (z_{1,+++}^2z_{2,+++} - z_{1,+++})
\end{array}
\]
is an algebra isomorphism. 
\end{example}

    \subsection{Gradings} Let $\word{i} = i_1\dots i_r$ be a braid word. Let $\phi \in 2^{\word{i}}$. We may assume that $\supp(\phi)$ is a subword with Demazure product $\dem(\word{i})$. Our goal is to give $\C[C_{\word{i}}(\phi)]$ a grading by the character lattice $\charlat(\dil{\word{i}})$ such that the (pullback of the) transition map $\C[C_{\word{i}}(\phi)][z_{t,\phi}^{-1} \mid \phi(t) = -] \to \C[C_{\word{i}}(\phi_{+})]_{loc}$ is a graded algebra isomorphism, where $loc$ denotes a localization of the algebra $\C[C_{\word{i}}(\phi_{+})] = \C[X(\word{i})]$. Note that the algebra $\C[X(\word{i})]$ is graded by $\charlat(\dil{\word{i}})$, and one of the challenges is to verify that this is also the case for the localizations of interest.

Let us assume, first, that there is a unique $t_0 \in [r]$ such that $\phi(t_0) = -$. By Corollary \ref{cor:not-essential}, we may assume that $t_0$ is not an essential crossing of $\word{i}$. 
We consider the word 
\[
\wword{i} = (i_1, \dots, i_{t_0-1}, i_{t_0}, i_{t_0}, i_{t_0+1}, \dots, i_{r})
\]
of length $r+1$, and consider the braid variety $X(\wword{i})$. 

\begin{notation}
We will denote the affine coordinates on $X(\wword{i})$ by $z_1, \dots, z_{r+1}$, while the coordinates on $X(\word{i}) = C_{\word{i}}(\phi_{+})$ will be denoted by $z_{1, \phi_{+}}, \dots, z_{r, \phi_{+}}$.
\end{notation}

By results of \cite[Section 2.3]{CGGS1}, we have
\[
\{(z_1, \dots, z_{r+1}) \in X(\wword{i})| z_{t_0+1} \neq 0 \} \cong \C^{\times} \times X(\word{i}).
\]
Let us elaborate on the explicit isomorphism. If $z_{t_0+1} \neq 0$ we have by an $\SL_2$-computation (cf. \cite[(11)]{CGGLSS})
\[
B^{+}_{i_{t_0}}(z_{t_0})B^{+}_{i_{t_0}}(z_{t_0+1}) = B^{+}_{i_{t_0}}(z_{t_0} - z_{t_0+1}^{-1})U_{i_{t_0}}(z_{t_0+1})
\]
and the $X(\word{i})$-part of the isomorphism is obtained by \lq\lq sliding the element $U_{i_{t_0}}(z_{t_0+1})$ to the right\rq\rq. In weave language, we picture this as follows:

\begin{equation}\label{eq:weave-boldfacebeta-to-beta-1}
    \begin{tikzpicture}
         \node at (-0.4, 2.2) {$\wword{i}$};
         \node at (-0.4, -0.2) {$\word{i}$};
         \draw[color=olive] (0,2) to (0,0);
         \draw[color= CadetBlue] (0.5, 2) to (0.5, 0);
         \node at (1, 0.8) {$\dots$};
         \draw[color=blue] (1.5,2) to[out=270, in=135] (2, 1) to (2, 0);
         \draw[color=blue] (2.5, 2) to[out=270, in=45] (2,1);
         \draw[color=magenta] (3.2,2) to (3.2,0);
         \draw[color=red] (3.9, 2) to (3.9, 0);
         \node at (4.5, 0.8) {$\dots$};
         \draw[color=teal] (5.1, 2) to (5.1, 0);
         \draw[dashed, color=YellowOrange] (2,1) to (5.5,1);
         \node at (0, 2.2) {\tiny{$z_1$}};
         \node at (0, -0.2) {\tiny{$z_1$}};
         \node at (0.5, 2.2) {\tiny{$z_2$}};
         \node at (0.5, -0.2) {\tiny{$z_2$}};
         \node at (1, 2.2) {$\cdots$};
         \node at (1.5, 2.2) {\tiny{$z_{t_0}$}};
         \node at (2.5, 2.2) {\tiny{$z_{t_0+1}$}};
         \node at (2.2, -0.2) {\tiny{$z_{t_0} - z_{t_0+1}^{-1}$}};
         \node at (3.2, 2.2) {\tiny{$z_{t_0+2}$}};
         \node at (3.2, -0.2) {\tiny{$z_{t_0+1}'$}};
         \node at (3.9, 2.2) {\tiny{$z_{t_0+3}$}};
         \node at (3.9, -0.2) {\tiny{$z'_{t_0+2}$}};
         \node at (4.5, 2.2) {$\cdots$};
         \node at (5.1, 2.2) {\tiny{$z_{r+1}$}};
         \node at (5.1, -0.2) {\tiny{$z'_{r}$}};
    \end{tikzpicture}
\end{equation}

On the other hand, the $\C^{\times}$-part of the isomorphism is given simply by the coordinate $z_{t_0+1} \neq 0$. In other words, we have the following algebra isomorphism:
\begin{equation}\label{eq:iso-localizations}
\begin{array}{rcl}
\C[X(\word{i})][z^{\pm 1}] & \xrightarrow{\cong} & \C[X(\wword{i})][z_{t_0+1}^{-1}] \\
z_{t, \phi_{+}} & \mapsto & z_{t} \; \text{if} \; t < t_0, \\
z_{t_0, \phi_{+}} & \mapsto & z_{t_0} - z_{t_0+1}^{-1}, \\
z_{t, \phi_{+}} & \mapsto & z_{t+1} \; \text{if} \; t > t_0, \\
z & \mapsto & z_{t_0+1}.
\end{array}
\end{equation}

\begin{lemma}\label{lem:iso-tori}
    The elements $z_{t_0} - z_{t_0+1}^{-1}, z_{t_0+1}', \dots, z_r' \in \C(X(\wword{i}))$ are homogeneous with respect to the $\dil{\wword{i}}$-action. Moreover, the map
    \[
    \begin{array}{rcl}
    \dil{\wword{i}} & \to & (\C^{\times})^{r} \\
    \xi & \mapsto & (\xi^{\wt(z_1)}, \xi^{\wt(z_2)}, \dots, \xi^{\wt(z_{t_0-1})}, \xi^{\wt(z_{t_0} - z_{t_0+1}^{-1})}, \xi^{\wt(z'_{t_0+1})}, \dots, \xi^{\wt(z_r')}),
    \end{array}
    \]
    induces an isomorphism $\dil{\wword{i}} \to \dil{\word{i}}$.
\end{lemma}
\begin{proof}
The first part of the statement follows from Lemma \ref{lem:3valent-dilation}. We only need to show that the map $\dil{\wword{i}} \to \dil{\word{i}}$ is an isomorphism.

First, we claim that $\dil{\word{i}}$ and $\dil{\wword{i}}$ are tori of the same rank, which is equivalent to saying that the cluster structures on $X(\word{i})$ and $X(\wword{i})$ have the same number of frozen variables. Note that, appending a weave for $\word{i}$ at the bottom of the weave \eqref{eq:weave-boldfacebeta-to-beta-1} we obtain a weave for $\wword{i}$. It follows that one obtains the ice quiver for $X(\wword{i})$ from the ice quiver in $X(\word{i})$ by adding a new vertex corresponding to the trivalent vertex in \eqref{eq:weave-boldfacebeta-to-beta-1}. We need to verify that this vertex is not frozen. Since $t_0$ is non-essential, the type ($\mathrm{A}_2$)-arrow $\wword{i} \to \word{i}$ is not frozen, so it follows from  Theorem~\ref{thm:cluster-structure} that the corresponding vertex in the quiver is not frozen, as needed.

It remains to show that the map is injective. Assume that $\xi \in \dil{\wword{i}}$ maps to the identity. In terms of flags, the map $X(\wword{i}) \dashrightarrow X(\word{i})$ is given simply by forgetting the flag $\flag{t_0}$, so the action of $\xi$ on an element $(\flag{0}, \dots, \flag{r+1}) \in X(\wword{i})$ fixes every flag except for, perhaps, the flag $\flag{t_0}$. But for any such $\xi$ we also have $\xi^{\wt(z_{t_0} - z_{t_0+1}^{-1})} = \xi^{\wt(z_{t_0})} = 1$, so the element must also fix the flag $\flag{t_0}$. Since $\dil{\wword{i}}$ acts faithfully on $X(\wword{i})$, the result follows. 
\end{proof}

We extend the action of $\dil{\word{i}}$ on $X(\word{i})$ to $\C^{\times} \times X(\word{i})$ by making $\xi \in \dil{\word{i}}$ act on $\C^{\times}$ by $\xi^{-\wt(z_{t_0, \phi_{+}})}$. Note that this makes the isomorphism \eqref{eq:index-set-charts} $\dil{\word{i}}$-equivariant, where we have used the isomorphism from Lemma \ref{lem:iso-tori} to make $\dil{\word{i}}$ act on the right-hand side. 

Thanks to Lemma \ref{lem:iso-tori}, the pullback of the functions $z_1, \dots, z_{r+1}$ to $X(\word{i})$ are $\dil{\word{i}}$-homogeneous. Computing, we have:
\[
z_t(z_{1, \phi_{+}}, \dots, z_{r, \phi_{+}}, z) = \begin{cases} z_{i, \phi_{+}} & \text{if} \; t < t_0 \\
z_{t_0, \phi_{+}} + z^{-1} & \text{if} \; t = t_0, \\ 
z & \text{if} \; t = t_0+1, \\
\widetilde{z}_{t-1, \phi_{+}} & \text{if} \; t > t_0+1\end{cases} 
\]
where $\widetilde{z}_{t-1, \phi_{+}}$ is given, up to signs, by sliding the element $U_{t_0}(-z^{-1})$ to the right, cf. Remark \ref{rmk:sliding-inverse}. Restricting to the $\dil{\word{i}}$-stable locally closed set $z_{t_0, \phi_{+}} = -z^{-1}$, we obtain that the elements $z_{t_0, \phi_{+}}^{-1}, z'_{t_0+1, \phi_{+}}, \dots, z'_{r, \phi_{+}}$ from \eqref{eq:transition-function} are homogeneous, as desired. Thus, we have the following result.

\begin{proposition}\label{prop:grading-one-minus}
    Assume that $\phi(t_0) = -$ but $\phi(t) = +$ for $t \neq t_0$. The isomorphism
    \[
    \C[C_{\word{i}}(\phi)][z_{t_0, \phi}^{-1}] \to \C[C_{\word{i}}(\phi_{+})][z_{t_0, \phi_{+}}^{-1}]
    \]
    endows $\C[C_{\word{i}}(\phi)][z_{t, \phi}^{-1}]$ with a $\charlat(\dil{\word{i}})$-grading so that each coordinate function $z_{1, \phi}, \dots, z_{r, \phi}$ is homogeneous. Moreover, $\C[C_{\word{i}}(\phi)]$ is a graded subalgebra of $\C[C_{\word{i}}(\phi)][z_{t, \phi}^{-1}]$.
\end{proposition}
\begin{proof}
    Only the fact that $\C[C_{\word{i}}(\phi)]$ is a subalgebra of $\C[C_{\word{i}}(\phi)][z_{t_0, \phi}^{-1}]$ does not follow from the discussion above, but this is a consequence of the fact that $C_{\word{i}}(\phi)$ is an irreducible variety and so $\C[C_{\word{i}}(\phi)]$ is an integral domain, cf. Proposition \ref{prop:irreducible}.
\end{proof}

Let us now deal with the case when $\phi$ takes the ``$-$'' value at an arbitrary number of elements in $[r]$. By Corollary \ref{cor:not-essential} we may and will assume none of these elements is an essential crossing of $\word{i}$. We can then consider the word $\wword{i}$ obtained by doubling $i_t$ for every $t \in [r]$ with $\phi(t) = -$, and the following weave:

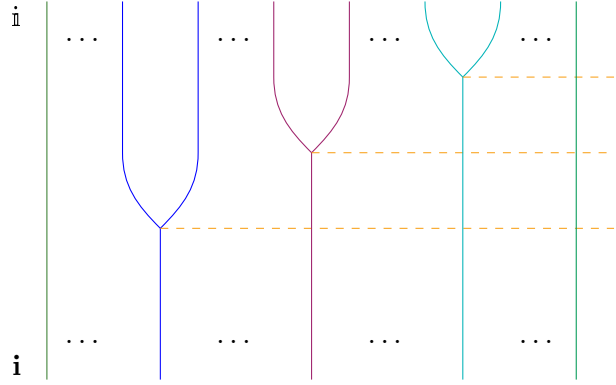
\begin{center}
    \begin{tikzpicture}
        \draw[color=OliveGreen] (0,0) to (0,5);
        \node at (0.5, 4.5) {$\dots$}; \node at (0.5, 0.5) {$\dots$};
        \draw[color=blue](1, 5) to (1,3) to[out=270, in=135] (1.5, 2) to (1.5,0); \draw[color=blue](2,5) to (2,3) to[out=270, in=45] (1.5,2);
        \draw[dashed, color=YellowOrange] (1.5, 2) to (7.5, 2);
        \node at (2.5, 4.5) {$\dots$}; \node at (2.5, 0.5) {$\dots$};
         \draw[color=RedViolet](3, 5) to (3,4) to[out=270, in=135] (3.5, 3) to (3.5,0); \draw[color=RedViolet](4,5) to (4,4) to[out=270, in=45] (3.5, 3);
         \draw[dashed, color=YellowOrange] (3.5, 3) to (7.5, 3);
          \node at (4.5, 4.5) {$\dots$}; \node at (4.5, 0.5) {$\dots$};
          \draw[color=BlueGreen](5, 5) to[out=270, in=135] (5.5,4) to (5.5,0); \draw[color=BlueGreen](6,5) to[out=270, in=45] (5.5,4);
          \draw[dashed,color=YellowOrange] (5.5,4) to (7.5, 4);
          \node at (6.5, 4.5) {$\dots$}; \node at (6.5, 0.5) {$\dots$};
          \draw[color=ForestGreen] (7,0) to (7,5);

          \node at (-0.4, 4.8) {$\wword{i}$};
          \node at (-0.4, 0.2) {$\word{i}$};         
    \end{tikzpicture}
\end{center}

Note that this weave identifies $(\C^{\times})^{\#\{t \mid \phi(t) = -\}} \times X(\word{i})$ with an open subset of $X(\wword{i})$. Moreover, since $\phi$ takes the value ``$+$'' at essential crossings, similarly to Lemma \ref{lem:iso-tori} we have an isomorphism of tori $\dil{\wword{i}} \cong \dil{\word{i}}$. Now it follows similarly by recursively applying Proposition \ref{prop:grading-one-minus} that $\C[C_{\word{i}}(\phi)][z_{t, \phi}^{-1} \mid \phi(t) = -]$ admits a grading by the character lattice $\charlat(\dil{\word{i}})$ in such a way that $z_{t, \phi}$ is homogeneous for every $t = 1, \dots, r$, compare to Example \ref{ex:running-transitions}. Thus, $\C[C_{\word{i}}(\phi)]$ inherits a grading by $\charlat(\dil{\word{i}})$. Summarizing, we have the following result.

\begin{theorem}\label{thm:action-extends}
    Let $\word{i}$ be a braid word. Then, the action of $\dil{\word{i}}$ on $X(\word{i})$ extends to an algebraic action on $\brick(\word{i})$. Moreover,
    \begin{enumerate}
        \item Each affine open set $C_{\word{i}}(\phi)$ is stable under the action of $\dil{\word{i}}$.
        \item For any subword $\word{j}\subset \word{i}$  with $\dem(\word{j}) = \dem(\word{i})$, the braid variety $X(\word{j}) \subseteq \brick(\word{i})$ is stable under the action of $\dil{\word{i}}$. 
    \end{enumerate}
\end{theorem}
\begin{proof}
We have compatible gradings by the character lattice $\charlat(\dil{\word{i}})$ on each coordinate algebra $C[C_{\word{i}}(\phi)]$, so the action of $\dil{\word{i}}$ on $\brick(\word{i})$ is obtained by gluing together actions on each affine set $C_{\word{i}}(\phi)$, and (1) is clear. For (2), recall from Remark \ref{rmk:equations-braid} that we can obtain the braid variety $X(\word{j})$ as a closed subset inside a single affine piece $C_{\word{i}}(\phi)$, given by the equations that some coordinates $z_{t, \phi}$ are zero. But by construction these coordinates are homogeneous with respect to the $\charlat(\dil{\word{i}})$-grading, so $X(\word{j})$ is indeed stable under the action of $\dil{\word{i}}$. 
\end{proof}

\begin{corollary}
\label{cor:bricks_subwrods_stable}
For any subword $\word{j}\subset\word{i}$ with $\dem(\word{j}) = \dem(\word{i})$, the brick variety $\brick(\word{j}) \subseteq \brick(\word{i})$ is stable under the action of $\dil{\word{i}}$.
\end{corollary}

\begin{proof}
This follows from Theorem~\ref{thm:action-extends} since $\brick(\word{j})$ is the union of braid varieties for subwords $\word{j}'$  of $\word{j}$ (and so of $\word{i}$) with $\delta(\word{j}') = \delta(\word{j})$ (and so with $\delta(\word{j}') = \delta(\word{i})$), cf. Theorem~\ref{thm:escobar_stratification}.
\end{proof}

Let us now  explicitly find the set $\brick(\word{i})^{\dil{\word{i}}}$. 

\begin{lemma}
\label{lem:fixed_points}
    There is a bijection between:
    \begin{enumerate}
        \item Points in $\brick(\word{i})$ that are fixed under the $\dil{\word{i}}$-action.
        \item Subwords $\word{j} \subseteq \word{i}$ such that $\pi(\word{j}) = \dem(\word{i})$. 
    \end{enumerate}
\end{lemma}
\begin{proof}
Let $\word{j} \subseteq \word{i}$ be as in (2), and let $\phi_{\word{j}} \in 2^{\word{i}}$ be as in Remark \ref{rmk:equations-braid}. The point in $C(\phi_{\word{j}})$ given by $z_{t, \phi_{\word{j}}} = 0$ for every $t$ belongs to $C_{\word{i}}(\phi_{\word{j}})$ thanks to our condition on $\word{j}$, and since each $z_{t, \phi_{\word{j}}}$ is homogeneous under the $\dil{\word{i}}$-action it is fixed under the action of $\dil{\word{i}}$. Note that, in terms of flags, this is the element all whose components are coordinate flags, and $\flag{t-1} = \flag{t}$ unless $t \in \word{j}$. 

It remains to verify that every element of $\brick(\word{i})^{\dil{\word{i}}}$ is of the form specified in the previous paragraph. For this, we use the action of the diagonal torus $T = \borel_{+} \cap \borel_{-}$. Note that the diagonal action of $T$ on $X(\word{i})$ is in fact the restriction of the diagonal action on $\brick(\word{i})$ which is, in turn, the restriction of the diagonal action on $\BS(\word{i})$. Moreover, recall that the diagonal action of $T$ on $X(\word{i})$ is by cluster dilations, so we have a map $a: T \to \dil{\word{i}}$, satisfying $a(t)(\flag{0}, \flag{1}, \dots, \flag{r}) = (t\flag{0}, t\flag{1}, \dots, t\flag{r})$ for every  $t \in T, \, (\flag{0}, \dots, \flag{r}) \in X(\word{i})$. By continuity, $a(t)(\flag{0}, \flag{1}, \dots, \flag{r}) = (t\flag{0}, t\flag{1}, \dots, t\flag{r})$ for every $t \in T$ and every element $(\flag{0}, \dots, \flag{r}) \in \brick(\word{i})$. Thus, a point fixed under the action of $\dil{\word{i}}$ must also be fixed under the diagonal action of $T$, so all its components are coordinate flags, cf. Lemma \ref{lem:fixed-points-braid}. 
\end{proof}

\begin{corollary}
\label{cor:invariant_divisors}
The set  
\begin{equation}
\label{eq:invariant_divisors}
\{ \overline{\brick(\word{i})^{\word{j}}} \cong \brick(\word{j}) \, | \, \word{j} \subset \word{i}, \ell(\word{j}) = \ell(\word{i}) - 1, \dem(\word{j}) = \dem(\word{i})\}
\end{equation}
provides an explicit set of $\dil{\word{i}}$-invariant divisors whose classes generate $\Pic(\brick(\word{i})) \cong \Cl(\brick(\word{i}))$.
\end{corollary}

\begin{proof}
Divisors in the set \eqref{eq:invariant_divisors} are $\dil{\word{i}}$-invariant by Corollary~\ref{cor:bricks_subwrods_stable}. The rest of the proof follows the logic of the proof of \cite[Lemma 5.29]{CGGLSS}; we provide the argument for completeness. Since $\brick(\word{i})$ is smooth and irreducible, its class and Picard groups are isomorphic, $\Pic(\brick(\word{i})) \cong \Cl(\brick(\word{i}))$.
Since both $\BS(\word{i})$ and $\brick(\word{i})$ are
smooth projective complex varieties which admit algebraic torus actions with finitely many fixed points, each of them admits a Białynicki-Birula decomposition and so its higher cohomology groups vanish. Thanks to the exponential sheaf sequence,  the first Chern class then identifies their Picard groups $\Pic(\BS(\word{i}))$ and $\Pic(\brick(\word{i}))$ with the second singular cohomology $H^2(\BS(\word{i}), \Z)$ and $H^2(\brick(\word{i}), \Z)$, respectively. By \cite[Corollary 6.5.(b)]{Haerterich} and \cite[Corollary 6.12]{Shchigolev}, the canonical restriction map $H^2(\BS(\word{i}), \Z) \to H^2(\brick(\word{i}), \Z)$ is surjective, and so the restriction map $\Pic(\BS(\word{i})) \to \Pic(\brick(\word{i}))$ is surjective. Consider the set 
\begin{equation}
\label{eq:BS_divisors}
\{ \BS(\word{j}) \, | \, \word{j} \subset \word{i}, \ell(\word{j}) = \ell(\word{i}) - 1 \}, 
\end{equation}
where $\BS(\word{j})$ by slight abuse of notation means here the subvariety of $\BS(\word{i})$ defined similarly to $\overline{\brick(\word{i})^{\word{j}}}$ as the closure of 
\[
\{(\flag{0}, \dots, \flag{r}) \in \BS(\word{i}) \mid \flag{t-1} \neq \flag{t} \; \text{if and only if} \; t \in \word{j}\}. 
\]
The set \eqref{eq:BS_divisors} is a basis of $\Pic(\BS(\word{i}))$, see e.g. \cite{LFT}. As observed in the proof of \cite[Lemma 5.29]{CGGLSS}, its elements restrict to $\overline{\brick(\word{i})^{\word{j}}}$ if $\dem(\word{j}) = \dem(\word{i})$ and to $0$ otherwise. So the restriction of the set \eqref{eq:BS_divisors} is precisely 
\eqref{eq:invariant_divisors}, and the surjectivity of the restriction implies that the latter set spans $\Pic(\brick(\word{i}))$.
\end{proof}

In order to speak of a moment map for the action of $\dil{\word{i}}$ on $\brick(\word{i})$, we need to choose a $\dil{\word{i}}$-linearizable ample line bundle. We do not check if the $T$-linearizable ample line bundle $L$ defining the projective embedding used by Escobar is  $\dil{\word{i}}$-linearizable, but since $\brick(\word{i})$ is an irreducible projective variety, at least some tensor power $L^{\otimes k}$ is, in a way compatible with the $T$-action, cf. e.g. \cite[Proposition 5.2.4]{Brion_linearization} and further references in \cite{Brion_linearization}. We choose such a power, and consider the corresponding moment map. The compatibility of Hamiltonian actions implies the following  statement; the dilation therein is precisely by the factor $k$. 

\begin{corollary}
\label{cor:proj_moment_to_brick_polytopes}
The image of $\brick(\word{i})$ under the moment map for the $\dil{\word{i}}$-action, which is a convex polytope by  \cite{Atiyah, GuilleminSternbergI},  projects onto a dilation of the brick polytope $B(\word{i}, \dem(\word{i}))$.
\end{corollary}

We can also prove the following result suggested in \cite[Remark 5.3]{EFMS}; see \cite{EFM, EFMS} for the definition of  cluster type pairs.

\begin{proposition}
\label{prop:brick_cluster_pair}
$(\brick(\word{i}), \brick(\word{i}) \backslash X(\word{i}))$ is a log Fano pair of cluster type.
\end{proposition}

\begin{proof}
The proof strategy is already outlined in \cite[Remark 5.3]{EFMS}. We only need to show that  $\brick(\word{i}) \backslash X(\word{i})$ supports an effective ample divisor and the class $[\brick(\word{i}) \backslash X(\word{i})]$ is anticanonical. The corresponding statements about the Bott-Samelson variety $\BS(\word{i})$ are explained in detail in \cite[Appendix A]{Anderson}, following \cite{LFT, MR}. The statements about $\brick(\word{i})$ now follow by taking the restriction to our fiber of the projection map $\BS(\word{i}) \to G/\borel_+$ as discussed above, since the divisors $X_i$ in the notation of \cite[Appendix A]{Anderson} restrict to the boundary divisors in the set \eqref{eq:invariant_divisors}, and restrictions of effective, ample, and (classes of) anticanonical divisors for such a fiber are effective, ample, and anticanonical. 
\end{proof}

\section{Toric brick varieties and subword complexes}\label{sec:toric-brick}

\subsection{Toric brick varieties}
\begin{corollary}
\label{cor:torus_implies_toric}
Let $\word{i}$ be a braid word, and assume the braid variety $X(\word{i})$ is a torus. Then, $\brick(\word{i})$ is a toric variety.
\end{corollary}

Note that the braid variety $X(\word{i})$ depends only on the underlying braid $\beta$ and not on the chosen braid word $\word{i}$, while $\brick(\word{i})$ is sensitive to the choice of braid word.

\begin{corollary}
\label{cor:subwords_tori_orbits}
Assume $\word{i}$ is such that for every subword $\word{j}$ of $\word{i}$ with the same Demazure product, the braid variety $X(\word{j})$ is a torus. 
Then, the $\dil{\word{i}}$-orbits in $\brick(\word{i})$ are in bijection with the subwords $\word{j}$ of $\word{i}$ with the same Demazure product.
\end{corollary}
\begin{proof}
    By Theorem \ref{thm:action-extends}(2), for each subword $\word{j}$ with $\dem(\word{j}) = \dem(\word{i})$ we have that $X(\word{j})$ is $\dil{\word{i}}$-stable, so it is a union of $\dil{\word{i}}$-orbits. It suffices to show that $X(\word{j})$ is a single orbit. Since $\brick(\word{i})$ is compact and $\dil{\word{i}}$-toric, it has a finite number of orbits, so the same is true for $X(\word{j})$. In particular, $X(\word{j})$ contains a dense orbit, that must be a torus, cf. e.g. \cite[Section 3.1]{Fulton}. Since $X(\word{j})$ is itself a torus, it must coincide with its dense orbit.
\end{proof}

In fact, if $X(\word{i})$ is a torus then for every subword $\word{j}$ of $\word{i}$ with the same Demazure product we automatically have that $X(\word{j})$ is a torus, as we show in Proposition \ref{prop:no mutables in subwords} below. 

\begin{lemma}\label{lem:criterion for mutables}
Let $\word{j}$ be a word that satisfies $\pi(\word{j})=w_0$ and $\word{j}$ is not reduced. Then $X(\word{j})$ has a mutable vertex.
\end{lemma}

\begin{proof}
Assume we can find braid moves that transform $\word{j}$ to $\word{j}'=\word{j}_1 (i i)\word{j}_2 $ with $\delta(\word{j}_1\word{j}_2)<\delta(\word{j}')=\delta(\word{j})$.
By \cite[Lemma 7.2]{CGGLSS} this implies that there is a frozen vertex associated to the reduction $\sigma_i\sigma_i\to \sigma_i$.
However, we get a contradiction as 
\[
w_0=\pi(\word{j})=\pi(\word{j}')=\pi(\word{j}_1\word{j}_2)\le \delta(\word{j}_1\word{j}_2).
\]
Therefore, in any weave for $\word{j}$ that starts with the braid moves turning $\word{j}$ to $\word{j}'$  the subsequent reduction $\sigma_i\sigma_i\to \sigma_i$ yields a mutable vertex.   
\end{proof}

\begin{proposition}\label{prop:no mutables in subwords}
 Let $\word{i}$ be a braid word with $\delta(\word{i})=w_0$ and $\word{j}$ a subword with $\delta(\word{j})=\delta(\word{i})$. 
 If $X(\word{i})$ has no mutable vertices then $X(\word{j})$  also has no mutable vertices. 
\end{proposition}

\begin{proof}
Let $\word{i}={i_1}\cdots {i_r}$ and $J=\{j_1,\dots,j_s\}\subset[r]$ be the support of the subword $\word{j}$, that is $\word{j}={i_{j_1}}\cdots {i_{j_s}}$. 
We proceed by induction on $\ell(\word{i})-\ell(\word{j})=r-s$.

If $r-s=1$ we may assume that $J=\{2,\dots,r\}$ using \cite[\S5.5]{CGGLSS} as follows: if $J\not=\{2,\dots,r\}$ let $[r]\setminus J=\{j\}$. Consider the cyclic rotation $\word{i}'$ of $\word{i}$ so that $\word{i}'={i_j} \word{j}'$.
By \cite[Theorem 5.31]{CGGLSS} there is a bijection between cluster variables of $X(\word{i})$ and $X(\word{i}')$ preserving mutable cluster variables.
As $\word{j}'$ is a consecutive subword of $\word{i}'$, if $X(\word{j}')$ has a mutable vertex obtained from a weave $\mathfrak W$, then we can extend $\mathfrak W$ to a weave for $\word{i}'$ which contains the same mutable vertex. This proves the claim for $r-s=1$.

If $r-s>1$ choose the first letter of $\word{i}$, say ${i_k}$ that is not a letter of $\word{j}$. Let $\word{j}'$ be the word obtained from $\word{i}$ by removing ${i_k}$. So the induction step applies to $\word{j}'\subset\word{i}$ and we deduce $X(\word{j}')$ has no mutable vertices. Moreover, $\ell(\word{j}')-\ell(\word{j})=r-s-1$ and the rest of the claim follows by induction.
\end{proof}

We can now give a characterization of those words $\word{i}$ for which $X(\word{i})$ is a torus. 

\begin{proposition}
\label{prop:torus_no_nonreduced}
$\word{i} = {i_1}\cdots {i_r}$ be a braid word such that $\delta(\word{i})=w_0$. Then the following are equivalent
\begin{enumerate}
    \item $X(\word{i})$ is a torus, \emph{i.e.} there are no mutable cluster variables;
    \item every subword $\word{j}$ of $\word{i}$ that satisfies $\pi(\word{j})=w_0$ is reduced.   
\end{enumerate}
\end{proposition}
\begin{proof} We show both implications.
\begin{itemize}
    \item[(1) $\Rightarrow$ (2)] 
Consider $\word{j}$ a subword of $\word{i}$ with $\pi(\word{j})=w_0$. We need to show that $\word{j}$ is reduced.
As $X(\word{i})$ has no mutable cluster variables, by Proposition~\ref{prop:no mutables in subwords} also $X(\word{j})$ does not have any mutable vertices.
Now Lemma~\ref{lem:criterion for mutables} implies that $\word{j}$ is reduced.

\item[(2) $\Rightarrow$ (1)] 
Assume that $\word{i}$  satisfies condition (2). Consider an arbitrary weave for $\word{i}$. Each braid move preserves the property (2) by Lemma~\ref{lem:braid_invariance_property_2}. Each trivalent vertex of the weave is Demazure frozen in the sense of \cite[Section 7.1]{CGGLSS}, and thus the corresponging cluster variable is frozen by \cite[Lemma 7.2]{CGGLSS}. Passing down through this trivalent vertex also trivially preserves the property (2). This shows that we can continue passing through the horizontal cross-sections from top to bottom until we reach the end of the weave, at each step the property will be satisfied, and each trivalent vertex would correspond to a frozen variable. So in the cluster seed for this weave, each variable is frozen, which implies that $X(\word{i})$ has no mutable variables.
\end{itemize}
\end{proof}

\subsection{Double root free words and polytopal realizations of subword complexes} In fact, the condition on a word $\word{i}$ in Proposition \ref{prop:torus_no_nonreduced} is equivalent to $\word{i}$ being double root free. The next three results on properties of words in Coxeter groups hold, with the same proofs, in larger generality than that of Weyl groups of simple algebraic groups. For possible future reference, we formulate and prove them in full generality.

\begin{lemma}
\label{lem:braid_invariance_property_2}
Let $(W, S)$ be a Coxeter system. Consider a group element $w \in W$ and two words of finite length related by a braid move  $\word{i} = \word{i}_1 (ij \ldots) \word{i}_2 \to \word{i}_1 (j i \ldots) \word{i}_2 = \word{i}'$.
Then $\word{i}$ does not contain any subwords which are non-reduced expressions of 

$w$
 if and only if $\word{i}'$ does not contain any subwords which are non-reduced expressions of 
 
 $w$.
 \end{lemma}

\begin{proof}
 If we have a non-reduced  subword for $w$ in $\word{i}'$ which is not a subword of $\word{i}$, there are two cases. Either it contains the entire consecutive subword $(ji \ldots)$, in which case the same positions in $\word{i}$ give a non-reduced  subword for $w$. 
 Or it contains a piece of $(ji \ldots)$ with some letters removed in the middle, which  has either some $ii$ with the letter $j$ between them omitted, or some $j j$ with the letter $i$ between them omitted. 
 If there are multiple such pairs, delete one, because it will still be a non-reduced subword for $w$. 
 This may create a new such pair, but the length of the subword will be shorter. 
 Repeat until there is exactly one such pair and after deleting it no new pair is created.  
 Except for this pair, the part of the remaining subword inside $(ji \ldots)$ has length at most $m_{ij} - 3$ and is either an initial or a final subword of $(i j \ldots)$. 
 Attach to it either the last and third to last letters of $(ij \ldots)$, if it was initial, or first and third letters of $(ij \ldots)$, if it was final.
We get a non-reduced subword for $w$ inside $\word{i}$. 
Thus, if the property did not hold for $\word{i}'$, it does not hold for $\word{i}$ either. The other implication follows by switching the roles of $\word{i}$ and $\word{i}'$.
\end{proof}

\begin{proposition}
\label{prop:no_nonreduced_double_root_free}
Let $(W, S)$ be a Coxeter system, $\word{i}$ a word of finite length. Then $\word{i}$ does not contain any subwords which are non-reduced expressions of $\dem := \dem(\word{i})$ if and and only if $\word{i}$ is double root free.
\end{proposition}

\begin{proof}
 If the subword $\word{j}$ has the same Demazure product as $\word{i}$, then each facet of the subword complex of $\word{j}$ extends to a facet of the subword complex of $\word{i}$ by adding letters in the complement $\word{i} \backslash \word{j}$. The root configuration for the former is a subset of the root configuration for the latter. Thus, $\word{i}$ is double root free if and only if each subword with the same Demazure product is double root free. The other condition (about non-reduced subwords), by 
 its definition, also holds for $\word{i}$ if and only if it holds for each subword of $\word{i}$ with the same Demazure product $\dem$.

 We now show that we can reduce the proof to the case of words of length $\ell(\dem) +2$. First, if $\word{i}$ is not double root free, the spherical subword complex $\Delta(\word{i}, \dem(\word{i}))$ has a facet $I$ having two (automatically flippable) positions $a, b$ with the same root in the root configuarion of $I$.  Then the subword $\word{j}$ of $\word{i}$ formed by $\word{i} \backslash I, a, b$ is not double root free.  It has length $\ell(\dem) + 2$, and $\dem(\word{j}) = \dem$. 
 Second, if $\word{i}$ does contain a non-reduced subword of $\dem$, then by the deletion property \cite[Proposition 1.4.7]{bb_coxeter}, we can remove pairs of letters of this subword until we get a subword of length $\ell(\dem) + 2$ still giving a non-reduced expression of $\dem$. Thus, each of the two conditions holds for $\word{i}$ if and only if it holds for each subword of $\word{i}$ of length $\ell(\dem) + 2$ with Demazure product $\dem$, and so it is sufficient to prove the proposition for words of length $\ell(\dem) + 2$.

Assume now that we have a word $\word{j}$ of length $\ell(\dem) + 2$ and with $\dem(\word{j}) = \dem$. 

Then $\word{j}$ has the form $\word{i}(u) i \word{i}(v) j \word{i}(w)$, with $\word{i}(u)\word{i}(v)\word{i}(w)$ being a reduced expression of $\dem$. Then the root configuration for the face of the subword complex which is a complement for this reduced expression is $\{u(\alpha_i), uv(\alpha_j)\}$. We have $\pi(\word{i}(u)i \word{i}(v) j \word{i}(w)) = \dem = \pi(\word{i}(u)\word{i}(v)\word{i}(w))$ if and only if $\pi(i \word{i}(v) j \word{i}(w)) = \pi(\word{i}(v)\word{i}(w))$, if and only if 
$\pi(i \word{i}(v) j) = \pi(\word{i}(v))$, if and only if $\pi(i \word{i}(v)) = \pi(\word{i}(v) j)$, if and only if $s_i (v (\alpha_j)) = v(s_j(\alpha_j))$, if and only if $s_i(v(\alpha_j)) = - v(\alpha_j)$, if and only if $v(\alpha_j) = \pm \alpha_i$, if and only if $uv(\alpha_j) = \pm u\alpha_i$.

Now that we proved that the failure of  first condition is equivalent to the word being ``not double root free, up to sign'', the result would follow if we prove that this implies the failure of the honest double root free condition.
 Assume that we have $\pi\left(\word{i}(u) i \word{i}(v) j \word{i}(w)\right) = uvw$, equivalently,  $v(\alpha_j) = \pm \alpha_i$. Since roots for the letters in the complement to our facet are in bijection with the inversions of $uvw$, 
 there is at most one letter in this complement which has the same root up to sign. Since the positions with letters $i, j$ are flippable, such letter must exist.  

This letter then can be replaced by $i$ and by $j$ by (two different) flips, and since no other flips of the maximal face $\{i, j\}$ are possible while the complex is spherical, the subword complex is a triangle formed by this letter $k$ and the positions of 
$i$ and $j$, i.e. the maximal simplices are $\{i, j\}, \{i, k\}, \{k, j\}$. 
By \cite[Lemma 2.6]{BergeronCeballos} (extending \cite[Lemmas 3.3 and 3.6]{cls} to possibly infinite Coxeter groups),
both roots in the facet formed by the two leftmost of these positions are positive, so the root configuration for  this leftmost (called \emph{positive greedy} in \cite{PilaudStump-EL}) facet 
consists of two copies of the same positive root. This completes the proof.
\end{proof}

\begin{corollary} 
\label{cor:drf_braid_invariant}
Let $(W, S)$ be a Coxeter system. Consider two words
of finite length related by a braid move $\word{i} = \word{i}_1 (ij \ldots) \word{i}_2 \to \word{i}_1 (j i \ldots) \word{i}_2 = \word{i}'$. Then $\word{i}$ is double root free if and only if $\word{i}'$ is double root free.
\end{corollary}

We now have all ingredients to prove the main result of this section.

\begin{theorem}\label{thm:braid-is-torus}
Assume that $\delta(\word{i}) = w_0$. The following are equivalent: 
\begin{itemize}
\item[(a)] $X(\word{i})$ is an algebraic torus;
\item[(b)] $\mathbb{C}[X(\word{i})]$ is a cluster algebra with no mutable vertices;
\item[(c)] For every subword $\word{j}$ of $\word{i}$ with $\delta(\word{j})= w_0$,  $X(\word{j})$ is an algebraic torus;
\item[(d)] For every subword $\word{j}$ of $\word{i}$ with $\delta(\word{j})= w_0$,  $\mathbb{C}[X(\word{j})]$ is a cluster algebra with no mutable vertices;
\item[(e)] $\brick(\word{i})$ is toric with respect to the extended action of the cluster dilation group $\dil{\word{i}}$;
\item[(f)] $(\brick(\word{i}), \brick(\word{i}) \backslash X(\word{i}))$ is a toric pair;
\item[(g)] $\word{i}$ does not contain any subword $\word{j}$ such that $\pi(\word{j}) = w_0$ and $\word{j}$ is not reduced;
\item[(h)] $\word{i}$ is double root free;
\item[(i)] For the extended action of the cluster dilation group $\dil{\word{i}}$ on $\brick(\word{i})$, for any $\dil{\word{i}}$-equivariant projective embedding of $\brick(\word{i})$, the polar dual of the moment polytope realizes the subword complex $\Delta(\word{i}, w_0)$.
\end{itemize}
\end{theorem}

\begin{proof}
Equivalences (a) $\Longleftrightarrow$ (b) and (c) $\Longleftrightarrow$ (d) are straightforward, see \cite[Corollary 2.2]{GKSB} for an explicit proof. Equivalence (b) $\Longleftrightarrow$ (d) follows from Proposition~\ref{prop:no mutables in subwords}.
Equivalence (a) $\Longleftrightarrow$ (g) is Proposition~\ref{prop:torus_no_nonreduced}. Equivalence (g) $\Longleftrightarrow$ (h) is Proposition~\ref{prop:no_nonreduced_double_root_free}. We have now proven the mutual equivalence of (a), (b), (c), (d), (g), and (h).

 We have $\dim \brick(\word{i}) = \dim X(\word{i})$, which is the number of all variables -- mutable and frozen,-- in each seed of the cluster structure on $\C[X(\word{i})]$, while the rank of  $\dil{\word{i}}$ equals the number of frozen variables. If $\brick(\word{i})$ is toric with respect to the extended action of $\dil{\word{i}}$, then  $\dim \brick(\word{i}) = \rank  \, \dil{\word{i}}$, which implies that all the variables are frozen. This shows the implication (e) $\Rightarrow$ (b). The implication (a) $\Rightarrow$ (e) is Corollary~\ref{cor:torus_implies_toric}. Thus, we have shown (a) $\Longleftrightarrow$ (e).

By definition of toric pairs, (f) is equivalent to $\brick({\word{i}})$ being toric for an action of some torus $\mathsf{T}$ and $X(\word{i})$ being a dense open orbit isomorphic to $\mathsf{T}$. So we have (f) $\Rightarrow$ (a). Since we know that (a) $\Longleftrightarrow$ (e) and when these conditions are satisfied,  $X(\word{i})$ is isomorphic to $\dil{\word{i}}$, we also get the  implication (a) $\Rightarrow$ (f).

The implication (c) $\Rightarrow$ (i) follows from Corollary~\ref{cor:subwords_tori_orbits} combined with Theorem~\ref{thm:escobar_stratification}. If (g) is not satisfied, then by Lemma~\ref{lem:fixed_points}, the moment polytope has strictly more vertices than the dual subword complex, so (i)  does not hold. This shows the implication (i) $\Rightarrow$ (g). Combined with an equivalence (c) $\Longleftrightarrow$ (g) we established three paragraphs ago, this completes the proof.
\end{proof}

\begin{remark}
\label{rem:brick_toric_alternative_proof}
The equivalence between (a)$\Longleftrightarrow$ (b), on one side, and (f), on the other side, can be alternatively deduced from Proposition~\ref{prop:brick_cluster_pair} combined with \cite{EFM}. Indeed, Proposition~\ref{prop:brick_cluster_pair} shows that the pair $(\brick(\word{i}), \brick(\word{i}) \backslash X(\word{i}))$ is a cluster type log-Calaby-Yau pair. The braid variety $X(\word{i})$ can be divisorially covered by a single almost torus if and only if it is a torus: the ``if'' direction is immediate from the definition of almost tori, and the ``only if direction'' follows e.g. from the fact a cluster algebra with $m$ variables in each seed which has at least one mutable variable cannot be isomorphic to the ring of regular functions of a big open subset of the algebraic torus of rank $m$. 

The desired equivalence then follows directly from \cite[Theorem 1.10]{EFM}. 
We refer the reader to \cite{EFM} for all the undefined notions.
\end{remark}

\begin{corollary}
\label{cor:polytope_projects_onto_brick}
Let $\word{i}$ be a double root free word with Demazure product $w_0$. There exists a polytopal realization of the subword complex of $(\word{i}, w_0)$ whose polar dual projects onto the brick polytope of  $\word{i}$. 
\end{corollary}

We finish this section with an application to the splicing maps of \cite{splicing2}.

\begin{remark}
    Assume that $\word{i}$ is a double root free word, so that $X(\word{i})$ is a torus. We also assume that $\dem(\word{i}) = w_0$. Consider an arbitrary decomposition of $\word{i}$  into a concatenation of two subwords $\word{i} = \word{i}_1\word{i}_2$. Choose some element $w \in W$ such that
    \[
    \dem(\word{i}_2\word{i}(w)) = w_0, \qquad \dem(\word{i}(w^{-1}w_0)\word{i}_1) = w_0.
    \]
    In this case, we have a \emph{splicing map}
    \[
    X(\word{i}(w^{-1}w_0)\word{i}_1) \times X(\word{i}_2\word{i}(w)) \to X(\word{i})
    \]
    whose image is a dense open set, cf. \cite[Theorem 1.1]{splicing2}. We claim that, in this case, the splicing map is an isomorphism. This confirms \cite[Conjecture 5.6]{splicing2} for such $\word{i}, \word{i}_1, \word{i}_2, w$.
    
    By \cite[Lemma 2.1]{splicing2} it is enough to show that both $X(\word{i}(w^{-1}w_0)\word{i}_1)$ and $X(\word{i}_2\word{i}(w))$ are tori. Applying cyclic rotations, we have that $X(\word{i}(w^{-1}w_0)\word{i}_1) \cong X(\word{i}_1\word{i}(w_0w^{-1}))$. Now, since $\dem(\word{i}_2\word{i}(w)) = w_0$, we must have $\dem(\word{i}_2) \geq w_0w^{-1}$, so that $\word{i}_2$ contains a reduced subword $\word{j}$ for $w_0w^{-1}$. We may apply braid moves in the isomorphism $X(\word{i}(w^{-1}w_0)\word{i}_1) \cong X(\word{i}_1\word{i}(w_0w^{-1}))$ to have that $X(\word{i}(w^{-1}w_0)\word{i}_1)$ is in fact isomorphic to $X(\word{i}_1\word{j})$. But $\word{i}_1\word{j}$ is a subword of $\word{i}_1\word{i}_2$, so $X(\word{i}_1\word{j})$ is a torus by Proposition \ref{prop:no mutables in subwords}. The proof that $X(\word{i}_2\word{i}(w))$ is a torus is similar. 
\end{remark}

\section{Applications to (projected) Richardson varieties}\label{sec:richardsons}

In this section, we explore applications of Theorems \ref{thm:action-extends} and \ref{thm:braid-is-torus} to the case of (closed) Richardson varieties in the flag variety. 

\subsection{Torus actions on Richardson varieties} We fix $v, w \in W$. By definition, the \emph{open Richardson variety} is the following locally closed subvariety of the flag variety $\group/\borel_{+}$:
\begin{equation}\label{eq:open-richardson}
\orich{v,w} = \{\borel \in \group/\borel_{+} \mid \borel_{+} \rel{w} \borel \rel{v^{-1}\wo} \borel_{-}\},
\end{equation}
while the \emph{closed Richardson variety} is the closure of $\orich{v,w}$ inside the flag variety
\[
\rich{v,w} := \overline{\orich{v,w}} \subseteq \group/\borel_+.
\]
The open Richardson variety $\orich{v,w}$ is nonempty if and only if $v \leq w$ in the (strong) Bruhat order. By definition, the same is true for the closed Richardson variety $\rich{v,w}$. In the case $v \leq w$, $\orich{v,w}$ is a smooth affine variety of dimension $\ell(w) - \ell(v)$. The variety $\rich{v,w}$ is a projective variety of dimension $\ell(w) - \ell(v)$, but it is no longer necessarily smooth.

For $q$ a prime power, we may consider the $\F_q$-points of the open Richardson variety, $\orich{v,w}(\F_q)$. By results of Deodhar, \cite{Deodhar}, the point-count is a polynomial in $q$ with integer coefficients, known as the $R$-polynomial \cite{KazhdanLusztig}
\[
\rpoly{v,w}(q) := \#\orich{v,w}(\F_q) \in \Z[q].
\]
In fact, the $R$-polynomial always has the form
\[
\rpoly{v,w}(q) = q^{\ell(w) - \ell(v)} - d_{v,w}q^{\ell(w) - \ell(v) - 1} + \text{lower order terms,}
\]
where $d_{v,w} \in \Z_{\geq 0}$ is the so-called \emph{Kazhdan-Lusztig $d$-invariant}, which is known to be a combinatorial invariant of the Bruhat interval $[v, w]$, cf. \cite{BarkleyGaetzLam, patimo}.

\begin{proposition}\label{prop:KL-d-invariant}
    Assume $v \leq w \in W$. The open Richardson variety $\orich{v,w}$ is isomorphic to the braid variety $X(\word{i}(w)\word{i}(v^{-1}w_0))$. Thus, it admits a cluster structure. The number of frozen variables in this cluster structure coincides with the Kazhdan-Lusztig $d$-invariant $d_{v,w}$. 
\end{proposition}
\begin{proof}
The first statement is \cite[Theorem 3.14]{CGGLSS}, and the second follows from the discussion in \cite[Section 3.4]{hypercubes}.
\end{proof}

Thanks to Proposition \ref{prop:KL-d-invariant}, the open Richardson variety $\orich{v,w}$ admits an action of a torus of rank $d_{v,w}$, that we denote $\dil{v,w}:= \dil{\word{i}(w)\word{i}(v^{-1}w_0)}$. Note that it also admits an action of the maximal torus $T = \borel_{+}\cap \borel_{-}$, and the action of $T$ extends to the closed Richardson variety $\rich{v,w}$. Moreover, the action of $T$ is by cluster dilations, so we have a map $T \to \dil{v,w}$. However, the $d$-invariant can easily exceed the rank of $T$ \cite{hypercubes}, so this map is not in general surjective. It is natural to ask whether the action of $\dil{v,w}$ on $\orich{v,w}$ extends to an action on $\rich{v,w}$, cf. \cite[Question 1.6]{GKSB}. The following result answers this in the affirmative. 

\begin{theorem}\label{thm:action-extends-richardsons}
    The faithful $\dil{v,w}$-action on the open Richardson variety $\orich{v,w}$ extends to an action on the closed Richardson variety $\rich{v,w}$. 
\end{theorem}
\begin{proof}
We consider the brick variety $\brick(\word{i}(w)\word{i}(v^{-1}w_0))$. This comes with a map 

\[
\varpi: \brick(\word{i}(w)\word{i}(v^{-1}w_0)) \to \group/\borel_{+}
\]
sending a tuple $(\flag{0}, \dots, \flag{\ell(w) + \ell(v^{-1}w_0)})$ to the flag $\flag{\ell(w)}$. According  to \cite[Theorem A.1]{KLS_projections}, the image of the map $\varpi$ is $\rich{v,w}$, the map $\varpi: \brick(\word{i}(w)\word{i}(v^{-1}w_0)) \to \rich{v,w}$ is birational, and it induces an isomorphism $\varpi: X(\word{i}(w)\word{i}(v^{-1}w_0)) \to \orich{v,w}$. Moreover, by \cite[Theorem A.3]{KLS_projections}, $\varpi^{\sharp}: \mathcal{O}_{\rich{v,w}} \to \varpi_{\ast}(\mathcal{O}_{\brick(\word{i}(w)\word{i}(v^{-1}w_0))})$ is an isomorphism \footnote{Such or similar desingularizations of closed Richardson varieties via brick varieties also appeared in \cite{Brion_flags, Balan, escobar}. We chose to cite \cite{KLS_projections} in the proof as the statements there are formulated in the most convenient form for our purposes.}. So we are in position to apply Blanchard's lemma \cite{Blanchard} as formulated by Brion in \cite[Theorem 7.2.1]{Brion_structure} to conclude that there is a $\dil{\word{i}(w)\word{i}(v^{-1}w_0)}$-action on $\rich{v,w}$ making the map $\varpi$ equivariant. Since the restriction $\varpi: X(\word{i}(w)\word{i}(v^{-1}w_0)) \to \orich{v,w}$ is an isomorphism, this is the desired action.
\end{proof}

We can explicitly describe the fixed points of this action.

\begin{proposition}\label{prop:fixed-points-richardsons}
For each $v \leq v' \leq w' \leq w$, the set $\rich{v',w'} \subseteq \rich{v,w}$ is invariant under the $\dil{v,w}$-action. Moreover,
\[
\rich{v,w}^{\dil{v,w}} = \rich{v,w}^{T} = \{v'\borel_{+} | v' \in [v,w]\}.\]
\end{proposition}
\begin{proof}
Let us show the first statement. It suffices to find a subvariety of $\brick(\word{i}(w)\word{i}(v^{-1}w_0))$ that maps onto $\rich{v',w'}$ under the map $\varpi$ and that is $\dil{v,w}$-invariant. Since $w' \leq w$, we can find a reduced word for $w'$ inside $\word{i}(w)$; since $v \leq v'$, we can find a reduced word for $(v')^{-1}w_0$ inside $\word{i}(v^{-1}w_0)$. Concatenating these words, we obtain a subword $\word{j}$ of $\word{i}(w)\word{i}(v^{-1}w_0)$. The variety $\brick(\word{j}) \hookrightarrow \brick(\word{i})$ is $\dil{v,w}$-stable by Corollary \ref{cor:bricks_subwrods_stable}. By construction, $\varpi(\brick(\word{j})) = \rich{v',w'}$. For the second statement, it follows similarly to the proof of Lemma \ref{lem:fixed_points} that every $\dil{v,w}$-invariant point is also $T$-invariant, so every $\dil{v,w}$-invariant point is indeed of the form $v'\borel_{+}$ for some $v' \in [v,w]$. To see that these points are indeed invariant, notice that $\{v'\borel_{+}\} = \rich{v',v'}$, which is $\dil{v,w}$-stable.
\end{proof}

A similar result holds for projections of Richardson varieties to parabolic flag varieties. More precisely, let $\borel_{+} \subseteq \parabolic$ be a parabolic subgroup of $\group$. A subvariety $\Pi \subseteq \group/\parabolic$ is called a \emph{projected Richardson variety} if there exist $v, w \in W$ such that $\Pi$ is the image of the Richardson variety $\rich{v,w}$ under the natural projection $\pi_{\parabolic}: \group/\borel_{+} \to \group/\parabolic$.

Now let $W_{\parabolic} \subseteq W$ be the associated parabolic subgroup. Given $v, w \in W$, following \cite[Section 2]{KLS_projections} we say that \emph{$w$ $\parabolic$-covers $v$}, and write $v \lessdot_{\parabolic} w$, if $w$ covers $v$ in the Bruhat order and $vW_{\parabolic} \neq wW_{\parabolic}$. Moreover, we denote by $\leq_{\parabolic}$ the transitive closure of the $\parabolic$-covering relation. Note that $\leq_{\parabolic}$ coarsens the usual Bruhat order. 

If $\Pi \subseteq \group/\parabolic$ is a projected Richardson variety, then there exist $v, w \in W$ such that $v \leq_{\parabolic} w$ and $\Pi = \pi_{\parabolic}(\rich{v,w})$, cf. \cite[Section 3]{KLS_projections}. In this case, we say that $\rich{v,w}$ is a \emph{Richardson model for} $\Pi$.

\begin{corollary}
    Let $\Pi \subseteq \borel/\parabolic$ be a projected Richardson variety with Richardson model $\rich{v,w}$. Then, there is a faithful $\dil{v, w}$-action on $\Pi$ making the projection $\pi_{\parabolic}: \rich{v,w} \to \Pi$ equivariant. 
\end{corollary}
\begin{proof}
    By \cite[Theorem 4.5]{KLS_projections}\footnote{See also \cite[Theorem 1.1]{BilleyCoskun} for a closely related statement with a very different proof.}, Blanchard's lemma \cite[Theorem 7.2.1]{Brion_structure} applies to the projection $\pi_{\parabolic}: \rich{v,w} \to \Pi$.
\end{proof}

\begin{remark}
Since the dimension of any compactification of $\orich{v,w}$ equals  $\ell(w) - \ell(v)$, we proved that all Richardson and projected Richardson varieties are $\dil{v, w}$-varieties of complexity
\[
\ell(w) - \ell(v) - d_{v,w}.
\]
\end{remark}

\subsection{Toric Richardson varieties} In this section, we extend \cite[Theorem 1.2]{GKSB} beyond type $A$, as follows.

\begin{theorem}\label{thm:toric-richardson-variety}
Let $v \leq w$ in Bruhat order. The following are equivalent.
\begin{itemize}
    \item[(a)] The open Richardson variety $\orich{v,w}$ is a torus.
    \item[(b)] The closed Richardson variety $\rich{v,w}$ is toric.
    \item[(c)] The Bruhat interval does not contain a subinterval isomorphic to $S_3$.
    \item[(d)] The Bruhat interval $[v,w]$ is a lattice.
    \item[(e)] The interval poset $\mathrm{Int}_{v,w} := \{[v',w'] \mid [v', w'] \subseteq [v,w]\}\cup\{\emptyset\}$, ordered by inclusion, is a lattice.
    \item[(f)] The Bruhat interval $[v,w]$ is the face lattice of a convex polytope.
    \item[(g)] The interval poset $\mathrm{Int}_{v,w}$ is the face lattice of a convex polytope.
    \item[(h)] Some (equivalently, each) word of the form $\word{i}(w)\word{i}(v^{-1}w_0)$ is double root free. 
    \item[(i)] The Kazhdan-Lusztig $d$-invariant is $d_{v,w} = \ell(w) - \ell(v)$.
\end{itemize}
\end{theorem}
\begin{proof}
(a) $\Longrightarrow$ (b) follows from Theorem \ref{thm:action-extends-richardsons}, while (b) $\Longrightarrow$ (c) $\Longrightarrow$ (a) follow exactly as in \cite[Lemma 3.5, Theorem 3.6]{GKSB}. The equivalence (c) $\Longrightarrow$ (d) $\Longrightarrow$ (e) $\Longrightarrow$ (c) also follows as in \cite{GKSB}, except that we need to show first that the face lattice of the moment polytope of the toric Richardson variety $\rich{v,w}$ is $\mathrm{Int}_{v,w}$. For every pair of elements $v \leq v' \leq w' \leq w$, the closed Richardson variety $\rich{v', w'} \subseteq \rich{v,w}$ is a $\dil{v,w}$-invariant set of $\rich{v,w}$, see \cite[Theorem 3.1]{Can-Toric}, so it follows that each open Richardson variety $\orich{v',w'}$ is also $\dil{v,w}$-stable. We need to show that it is, in fact, a single $\dil{v,w}$-orbit. For this, it is enough to find a $\dil{v,w}$-orbit in $\brick(\word{i}(w)\word{i}(v^{-1}w_0))$ that maps onto $\orich{v', w'}$. Since $w' \leq w$, we can find a reduced expression for $w'$ as a subword of $\word{w}$ and, since $v \leq v'$, we can find a reduced expression for $(v')^{-1}w_0$ as a subword of $\word{i}(v^{-1}w_0)$. We concatenate these two subwords to find a subword $\word{j} \subseteq \word{i}(w)\word{i}(v^{-1}w_0)$ such that $X(\word{j})$ maps onto $\orich{v',w'}$. Since $X(\word{j})$ is a single $\dil{v,w}$-orbit, the result follows. Note that this also proves (a) $\Longrightarrow$ (g), while (g) $\Longrightarrow$ (c) follows because the interval poset of $S_3$ is not a lattice. The equivalence of (a) and (h) follows directly from Theorem \ref{thm:braid-is-torus}. 
It remains to show the equivalence of (f) with the other conditions. Note that the Bruhat interval $[v,w]$ is isomorphic to the interval $[[v,v], [v,w]]$ in $\mathrm{Int}_{v,w}$, so (g) $\Longrightarrow$ (f). On the other hand, if the Bruhat interval $[v,w]$ is the face lattice of a polytope then it cannot have subintervals isomorphic to $S_3$, as $S_3$ is not a lattice. So (f) $\Longrightarrow$ (c) and this finishes the proof. Finally, $\dim(\orich{v,w}) = \ell(w) - \ell(v)$, and the number of frozen variables is exactly $d_{v,w}$. So the equivalence of (a) and (i) is a special case of the equivalence between (a) and (b) in Theorem \ref{thm:braid-is-torus} above.
\end{proof}

\begin{remark}
The equivalence (a) $\Longleftrightarrow$ (b) can be alternatively deduced from \cite[Theorem 1.7]{EFMS} combined with \cite[Theorem 1.10]{EFM} in a similar way to the argument for brick varieties outlined in Remark~\ref{rem:brick_toric_alternative_proof}.
\end{remark}

\begin{corollary}
The Bruhat interval $[v,w]$ is the face lattice of a convex polytope if and only if some (equivalently, each) word of the form $\word{i}(w)\word{i}(v^{-1}w_0)$ is double root free. In this case, such a polytope can be given explicitly as a vertex figure of the moment polytope (with respect to the $\dil{v,w}$-action) of $\rich{v,w}$.
\end{corollary}

We note that, given a pair $v \leq w$, it is significantly easier and faster to verify whether some (and by Corollary~\ref{cor:drf_braid_invariant}, equivalently, each) word of the form $\word{i}(w)\word{i}(v^{-1}w_0)$ is double root free or not than to check if the Bruhat interval does (not) contain a subinterval isomorphic to $S_3$. The full algorithm requires picking a pair of reduced expressions of $w$ and $v^{-1}w_0$, checking on all $\binom{\ell(w) - \ell(v) + \ell(w_0)}{\ell(w_0)}$ subwords of length $\ell(w_0)$ of the chosen word of the form $\word{i}(w)\word{i}(v^{-1}w_0)$ whether they are reduced expressions of $w_0$, and checking if there are repetitions in the corresponding root configurations.
In fact, a check of the double root free property is incorporated in Sage as a single command. Thanks to Proposition~\ref{prop:no_nonreduced_double_root_free} and its proof, one can alternatively check whether the word $\word{i}(w)\word{i}(v^{-1}w_0)$ contains any subwords $\word{j}$ of length $\ell(w_0) +2$ with $\pi(\word{j}) = w_0$. 

For $S_3$ itself, one possible word for $v = e, w_0 = s_1 s_2 s_1$ is $121121$. The first three positions form a facet of the subword complex $\Delta(121121, w_0)$, and the roots for the first and third indices both equal $\alpha_1$. This illustrates the failure of the double root free property for this word. It is also immediate to check directly that the cluster structure on $\C[X(121121)]$ has mutable variables: we can start a weave with a trivalent vertex $121121 \to 12121$, and this vertex is clearly mutable.

All of this shows that having an equivalence (2) $\Longleftrightarrow$ (8) in Theorem \ref{thm:toric-richardson-variety} should prove very useful in the search and study of explicit examples of toric closed Richardson varieties.

\bigskip

We finish this section with a couple of observations.

\begin{remark}
Fix a word $\word{j}_w \word{j}_{v^{-1}w_0}$ of the form $\word{i}(w)\word{i}(v^{-1}w_0)$.
The map $\varpi: \brick(\word{j}_w \word{j}_{v^{-1}w_0}) \to \rich{v,w}$ is a $\dil{v,w}$-equivariant resolution of singularities, and we now know that in case $\orich{v,w}$ is a torus, both the domain and codomain of $\varpi$ are $\dil{v,w}$-toric varieties. This implies that, in this case, the normal fan of the moment polytope of $ \brick(\word{j}_w \word{j}_{v^{-1}w_0})$ is simplicial and refines the  normal fan of the moment polytope of $\rich{v,w}$.
The cones of these fans bijectively correspond to faces of the respective polytopes, and we can see what happens on the level of face lattices. It is straightforward to deduce from \cite[Lemma 2.1]{splicing2} that faces of the moment polytope of $ \brick(\word{j}_w \word{j}_{v^{-1}w_0})$, equivalently, faces of the dual subword complex of $(\word{j}_w \word{j}_{v^{-1}w_0}, w_0)$, are in bijection with pairs (a subword of $\word{j}_w$ with Demazure product $w'$, a subword of $\word{j}_{v^{-1}w_0}$ with Demazure product $((v')^{-1} w_0)$), with $v \leq v' \leq w' \leq w$. At the same time, the faces of the moment polytope $\rich{v,w}$ are in bijection simply with pairs $(v',w')$, with $v \leq v' \leq w' \leq w$.
We see that the first set of pairs of course refines the second one. The reader should keep in mind that the first set, just as the brick variety itself, does depend nontrivially on the choice of word  $\word{j}_w \word{j}_{v^{-1}w_0}$ of the form $\word{i}(w)\word{i}(v^{-1}w_0)$ and not just on the pair $(v, w)$, cf. Examples~\ref{ex:Coxeter} and \ref{ex:4-crown}.
\end{remark}

\begin{remark}
The moment polytope of $\rich{v,w}$ for the $T$-action is called the \emph{Bruhat interval polytope} $Q_{v,w}$ \cite{KW, TW}. Its face lattice in general does not coincide with the interval poset $\mathrm{Int}_{v,w}$. We note that since the actions of $T$ and $\dil{v,w}$ are compatible, in the $\dil{v,w}$-toric case there exists a polytope with the face lattice given by  $\mathrm{Int}_{v,w}$ which projects onto $Q_{v,w}$: namely, the moment polytope of $\rich{v,w}$ for the action of $\dil{v,w}$, up to possible dilation (as in Corollary~\ref{cor:proj_moment_to_brick_polytopes}). We also note that in recent work \cite{GU}, a different projective embedding of brick varieties than that of Escobar is considered, with a different moment polytope realizing the same normal fan, and such that in the case of root independent words of the form of the form $\word{i}(w)\word{i}(v^{-1}w_0)$, this moment polytope contains $Q_{v,w}$ as a Minkowski summand.
\end{remark}

\section{Examples}
\label{sect:examples}

\begin{example}
\label{ex:Coxeter}
Consider the lower interval $[e, w]$, for some $w \leq c$, where $c$ is some Coxeter element in $W$. In this case, the variety $\rich{e,w}$ is smooth and $T$-toric, and the map $T \to \dil{e,w}$ is  surjective. Thus, $Q_{e,w}$ coincides with the moment polytope of $\rich{e,w}$ for the $\dil{e,w}$-action. The interval $[e, w]$ is Boolean, and so we see that the $Q_{e,w}$, having $\mathrm{Int}_{e,w}$ as its face lattice, is combinatorially isomorphic to an $\ell(w)$-dimensional cube. Any vertex figure of a combinatorial cube is a combinatorial simplex, in particular, $[e,w]$ is the face lattice of such a combinatorial simplex. 

Each of the braid words $\word{j}_w \word{j}_{w_0}$ of the form $\word{i}(w)\word{i}(w_0)$ is double root free. In fact, these are prototypical examples of root indepentent words. 

    Consider $w = c$. In this case, the map $T \to \dil{e,w}$ is an isomorphism. For each braid word $\word{j}_c \word{j}_{w_0}$ of the form $\word{i}(c)\word{i}(w_0)$, the map $\varpi: \brick(\word{j}_c \word{j}_{w_0}) \to \rich{v,w}$ is a $T \cong \dil{e,c}$-equivariant desingularization, and the normal fan of the moment polytope $\brick(\word{j}_c \word{j}_{w_0})$ refines that for $\rich{v,w}$. Depending on the choice of the word, we may obtain different varieties and combinatorially different polytopes and normal fans. 

As one extremum, we consider the reversed word of $\word{j}_c$ and denote it by $\word{j}_{c^{-1}}$ (since it is of the form  $\word{i}(c^{-1})$). Denote by $\word{j}_{w_0(c^{-1})}$ the lexicographically first reduced subword for $w_0$ in the infinite word $\word{j}_{c^{-1}}\word{j}_{c^{-1}}\word{j}_{c^{-1}}\ldots$; it is called the \emph{$\word{j}_{c^{-1}}$-sorting word for $w_0$}. The subword complex for $\word{j}_c \word{j}_{w_0(c^{-1})}$ is isomorphic to the boundary complex of the cross-polytope, and the moment polytope of $\brick(\word{j}_c \word{j}_{w_0(c^{-1})})$ is a combinatorial cube. We see that for this choice of word, the map $\varpi$ is in fact an isomorphism and the polytope  $\brick(\word{j}_c \word{j}_{w_0(c^{-1})})$ coincides with $Q_{e,c}$.

If we now consider the \emph{$\word{j}_{c}$-sorting word for $w_0$} denoted $\word{j}_{w_0(c)}$, i.e. the lexicographically first reduced subword for $w_0$ in the infinite word $\word{j}_{c}\word{j}_{c}\word{j}_{c}\ldots$, then the subword complex is isomorphic to the \emph{$c$-cluster complex of $W$}, and the moment polytope of $\brick(\word{j}_c \word{j}_{w_0(c)})$ is (up to translation) the \emph{$c$-associahedron of $W$}. In our crystallographic case, these are the  cluster complex and generalized associahedron of the cluster algebra of the orientation of $\dynkin$ induced by $c$. The normal fan of this polytope is the $g$-vector fan of this cluster algebra. We refer the reader to \cite{pilaud_stump_15, JLS, PSZ} for details and further references.

Between these two extremal cases, we have moment polytopes for other words of the form $\word{j}_c\word{i}(w_0)$. For each of them, the normal fan refines the normal fan of $Q_{e,w}$, i.e. of the moment polytope of $\brick(\word{j}_c \word{j}_{w_0(c^{-1})})$. Each braid move between two such words inside their $\word{i}(w_0)$-parts corresponds, in one direction of the move,  either to an isomorphism or to a sequence of edge subdivisions of spherical subword complexes. On the level of normal fans of moment polytopes, this corresponds to a sequence of simplicial subdivisions of a single cone induced by adding new rays. The partial order on a quotient of the set of commutation classes of words of the form $\word{i}(w_0)$, where cover relations are braid moves inducing subdivisions, can be seen as a 2-dimensional version of the Cambrian order of Reading and is isomorphic to some of the orders introduced in \cite{GW1, GW2}. Some parts of this were sketched in the short note \cite{G_subword_1} and details and corrections should appear in \cite{GW_upcoming}.
\end{example}

\begin{example}
\label{ex:4-crown}
Consider $W = A_3$ and let $v = s_1s_3, w = s_1s_2s_3s_2s_1 = s_1s_3s_2s_3s_1 = s_1s_3s_2s_1s_3$. The interval $[v,w]$ is a lattice, but it is not Boolean: it is a so-called 4-crown. The words $\word{i} = (12321)(2132)$ and $\word{i'} = (13213)(2132)$ of the form 
$\word{i}(w)\word{i}(v^{-1}w_0)$ are root independent. The Richardson variety $\rich{v,w}$ is therefore $T$-toric, but it is not smooth. The corresponding moment polytope is drawn in \cite[Figure 3]{GKSB}. One can see that it is not simple: while vertices 1 -- 8 are incident to 3 edges, the vertices 0 and 9, corresponding to intervals $[v,v]$ and $[w,w]$ in the interval poset $\mathrm{Int}_{v,w}$, are incident to 4 edges each. It has 8 facets (3 quadrilateral facets containing the vertex 4 are not visible). 

The subword complex $\Delta(\word{i}, w_0)$ has 8 vertices: the complement to the 3rd position in $\word{i}$ does not contain a reduced expression of $w_0$, the complements to all the other positions do. The brick polytope for $\word{i}$, which is the moment polytope of $\brick(\word{i})$ for $T \cong \dil{v,w}$-action, has therefore 8 facets, and its normal fan has 8 rays, just like the normal fan of the moment polytope of $\rich{v,w}$. Thus, in this case, the refinement of the fans is a simplicial refinement without adding new rays. On the level of polytopes, the moment polytope of $\brick(\word{i})$ can be obtained from the moment polytope of 
$\rich{v,w}$ by translating some facet hyperplanes parallel to themselves in such a way that each of the vertices 0 and 9 gets ``split'' into two trivalent vertices $0', 0''$ and $9', 9''$, respectively, with $0'$ and $0''$ being connected by an edge and $9'$ and $9''$ being connected by an edge. 

The word $\word{i'}$ is of the form $\brick(\word{j}_c \word{j}_{w_0(c)})$ for $c = s_1 s_3 s_2$, so the subword complex is the cluster complex for the type $A_3$ quiver with bipartite orientation $1 \to 2 \leftarrow 3$ and the moment polytope of $\brick(\word{i'})$ is the generalized associahedron for this quiver. This polytope has 9 facets, so in this case the refinement of the normal fan of the moment polytope of $\rich{v,w}$ adds one new ray. 
\end{example}

\begin{example}\label{ex:hypercubes}
An infinite family of examples comes from the \lq\lq big Bruhat hypercubes\rq\rq \, of \cite{hypercubes}. Fixing $n$, \cite{hypercubes} constructs a family of permutations $x_n, y_n \in S_{2^{n}}$. The Bruhat interval $[x_n, y_n]$ is Boolean, so $\rich{x_n, y_n}$ is $\dil{x_n, y_n}$-toric and its moment polytope is combinatorially a hypercube. Note that the rank of $\dil{x_n, y_n}$ is $n2^{n-1}$, which is larger than $2^{n}-1$ for $n \geq 2$ (and in fact $\rank(\dil{x_n, y_n})/\rank(T) \sim O(\log n)$) so $\rich{x_n, y_n}$ is not $T$-toric for $n \geq 2$.  Note that this example is also considered in \cite[Section 5.3]{GKSB}, where the authors construct a sequence of plabic graphs associated to the Richardson variety $\rich{x_n, y_n}$. 

The permutations $x_n, y_n$ are both involutions and $x_ny_n = y_nx_n = \wo$, cf. \cite[Lemma 2.8]{hypercubes}, so that $y_n = x_n^{-1}\wo$. Thus, we can take two (not necessarily distinct) reduced words for $y_n$: $\word{i}(y_n), \word{j}(y_n)$, and we obtain that the word $\word{i}(y_n)\word{j}(y_n)$ is double-root free and the brick variety $\brick(\word{i}(y_n)\word{j}(y_n))$ is $\dil{\word{i}(y_n)\word{j}(y_n)}$-toric.
\end{example}

\begin{example}
\label{ex:even_odd}
Another infinite family of toric Richardson varieties is discussed in \cite[Section 5.2]{GKSB}, see also \cite[Remark 5.18]{splicing2}. Here, we take $n = 2k$ even, $v_k = s_2s_4\cdots s_{n-2}$ and $w_k = (1n) = s_1\cdots s_{n-2}s_{n-1}s_{n-2}\cdots s_{1}$. In \cite[Remark 5.18]{splicing2} it is verified that $\orich{v, w}$ is a torus of rank $(n-1)+ (k-1)$, so the Richardson variety is $\dil{v,w}$-toric but not $T$-toric for $k \geq 2$.  Relating to Example \ref{ex:hypercubes}, we have that $x_2 = v_2, y_2 = w_2$. For $k \geq 3$, the Bruhat interval is a non-Boolean lattice, so the Richardson variety $\rich{v_k, w_k}$ is toric but not smooth. 

Note that $\word{i}(w_k) = (1,\dots, n-2, n-1, n-2, \dots, 1)$. Let $\mathsf{even} = (2,4,\dots, 2k-2)$ and $\mathsf{odd} = (1, 3, \dots, 2k-1)$. A reduced word for $\wo$ is $(\mathsf{even},\mathsf{odd})^{k}$, so a reduced word for $v_k^{-1}\wo$ is $\mathsf{odd}(\mathsf{even},\mathsf{odd})^{k-1}$. Thus, the subword complex of the word
\[
(1, \dots, n-2, n-1, n-2, \dots, 2, 1, \mathsf{odd}, (\mathsf{even}, \mathsf{odd})^{k-1}) 
\]
is realizable by a polytope. For example, for $k = 3$ we obtain the word
\[
(1,2,3,4,5,4,3,2,1,1,3,5,2,4,1,3,5,2,4,1,3,5).
\]
\end{example}

\begin{example}
\label{ex:duplicated}
The first family of toric brick varieties which do not appear as resolutions of toric Richardson varieties comes from the duplicated words $i_1 i_1 i_2 i_2 \ldots i_\ell i_\ell$, where $i_1 i_2 \ldots i_\ell$ is of the form $\word{i}(w_0)$, or subwords of these where we doubled only some of the letters. If we doubled at most $n$ letters whose roots were linearly independent, the subword complex is root independent and isomorphic to the boundary complex of a cross-polytope, and so it is realized by the polar dual of the brick polytope which is a parallelepiped, cf. \cite[Example 3.9, Example 4.3]{pilaud_stump_15}. If the word $\word{i}$ is obtained by doubling an arbitrary subset of cardinality $k$ of letters in $i_1 i_2 \ldots i_\ell$, the word is no longer necessarily root independent, but it is always double root free, cf. \cite[Section 5.2.2]{PilaudStump-EL} and Example~\ref{ex:112211}. The subword complex is still isomorphic to the boundary complex of the $k$-dimensional cross-polytope, and it is realized by the polar dual of the moment polytope of $\brick(\word{i})$. The latter projects onto the lower-dimensional brick polytope $B(\word{i},w_0)$. 
\end{example}

\begin{example}
In general, there are many toric brick varieties which do not appear as resolutions of toric Richardson varieties. Already for $W = A_2$, further examples are given by the words $122112$ and $211221$.

In $A_3$, root independent words have length at most 9, and one can find many double root free words of length 10, among which all but few are not resolutions of toric Richardson varieties. As a concrete example, consider the word $\word{i} = 1212132211$. It is double root root free, non-Richardson, and does not fit into the setting of Example~\ref{ex:duplicated}. Every reduced expression of $w_0$ inside $\word{i}$ contains position 6 which is essential, 
exactly one of the positions 7 and 8, exactly one of the positions 9 and 10, and exactly three consecutive modulo 5 positions out of indices 1--5. Combinatorially, the 4-dimensional polytope whose polar dual realizes the subword complex $\Delta(\word{i}, w_0)$ is the product of a pentagon (whose dual realizes $\Delta(12121, s_1s_2s_1)$) and a parallelogram (whose dual realizes $\Delta(2211, s_2s_1)$).
\end{example}

To finish, we discuss nice families of varieties which are neither $T$-toric, nor $\dil{\word{i}}$-toric. Our general results still apply, so they admit a faithful action of a large torus whose rank is given by the number of frozen variables.

\begin{example}
\label{ex:schubert_not_coxeter}
Consider the lower interval $[e, w]$, for some $w$ whose reduced expressions contain repeated letters, equivalently, for $w$ which does not fit in the generality of Example~\ref{ex:Coxeter}. In this case, for each word $\word{j}_w \word{j}_{w_0}$ of the form $\word{i}(w)\word{i}(w_0)$, positions in the subword $\word{j}_w$ form a facet of $\Delta(\word{j}_w \word{j}_{w_0}, w_0)$ and roots in its root configuration are just the simple roots corresponding to the letters in this subword.  Since $\word{j}_w$ has repeated letters, this root configuration has repetitions, and so $\word{j}_w \word{j}_{w_0}$ is not double root free. Thus, neither the brick variety $\brick(\word{j}_w \word{j}_{w_0})$, nor the Richardson variety $\rich{e,w}$ are toric.

In fact, the open Richardson variety $\orich{e,w} \cong X(\word{j}_{w}\word{j}_{w_0})$ is an example of a \emph{double Bott-Samelson variety} \cite{SW}, see \cite[Lemma 3.16]{CGGLSS}. Similarly to the type $A$ case proved in \cite[Lemma 4.41]{CGSS}, the map $T \to \dil{e,w}$ is surjective, so that $\orich{e,w}$ is a torus if and only if $\rich{e,w}$ is toric, if and only if $\rich{e,w}$ is in fact toric with respect to the usual $T$-action. Moreover, the variety $X(\word{j}_w\word{j}_{w_0})$ has only frozen variables if and only if $\word{j}_w$ does not have repeated letters, cf. \cite[Proposition 4.44]{CGGLSS}. So $\rich{e,w}$ is toric if and only if $w \leq c$ for some Coxeter element $c$ as in Example~\ref{ex:Coxeter}, and we recover the main result of \cite{Karuppuchamy}.  
\end{example}

In Examples~\ref{ex:Coxeter} and \ref{ex:schubert_not_coxeter} we see that Richardson varieties of the form $\rich{e,w}$, which are in fact the same as \emph{Schubert varieties} $X_w$ in $\group/\borel_+$, are toric if and only if they are T-toric.
Examples~\ref{ex:hypercubes} and \ref{ex:even_odd} illustrate that this is not the case for general Richardson varieties.

\begin{example}
The second family of non-toric varieties is defined by slightly modifying Example \ref{ex:hypercubes}, as in \cite[Section 3.2]{hypercubes}.
Namely, consider $p \geq 2$ and $n \geq 1$. Then we have permutations $x_n^{(p)}, y_n^{(p)} \in S_{p^n}$ such that, as posets
\[
[x_n^{(p)}, y_n^{(p)}] \cong (S_p)^{np^{n-1}}.
\]
Example \ref{ex:hypercubes} is precisely the case $p = 2$. For $p > 2$, the poset $(S_p)^{np^{n-1}}$ does contain intervals isomorphic to $S_3$, so that the open Richardson variety $\orich{x_n^{(p)}, y_n^{(p)}}$ is not a torus and the Richardson variety $\rich{x_n^{(p)}, y_n^{(p)}}$ is not toric. Nevertheless, the $d$-invariant is
\[
d_{x_n^{(p)}, y_n^{(p)}} = (p-1)np^{n-1} > p^{n} - 1 \; \text{if $n \geq 2$},
\]
so that for $n \geq 2$, the variety $\rich{x_n^{(p)}, y_n^{(p)}}$ admits a faithful algebraic action of the torus $\dil{x_n^{(p)}, y_n^{(p)}}$ of rank strictly greater than the rank of the maximal torus in $\SL_{p^n}$. 
\end{example}

\bibliographystyle{alpha}
\bibliography{bibliography}
\end{document}